\documentclass[11pt]{amsart}
\usepackage[english]{babel}
\usepackage{amsmath}
\usepackage{amssymb}
\usepackage{amsthm}
\usepackage{libertine}
\usepackage[all]{xy}
 \newtheorem{thm}{Theorem}[section]
 \newtheorem{cor}[thm]{Corollary}
 \newtheorem{lem}[thm]{Lemma}
 \newtheorem{prop}[thm]{Proposition}
  
 \theoremstyle{definition}
 \newtheorem{defn}[thm]{Definition}
 \theoremstyle{remark}
 \newtheorem{rem}[thm]{Remark}
 
 \numberwithin{equation}{section}
\newcommand{\norm}[1]{\left\Vert#1\right\Vert}
\newcommand{\scal}[1]{\left<#1\right>}
\newcommand{\Hq}{\mathbb H}
\newcommand{\Sq}{\mathbb S}

\newcommand{\R}{\mathbb{R}}      

\newcommand{\C}{\mathbb{C}}

\title[On the polyanalytic short-time Fourier transform in the quaternionic setting]{On the polyanalytic short-time Fourier transform in the quaternionic setting}

\begin{document}
\date{}
\author{Antonino De Martino, Kamal Diki}
\maketitle
\begin{abstract}
In this paper, we consider a quaternionic short-time Fourier transform (QSTFT) with normalized Hermite functions as windows. It turns out that such a transform is based on the recent theory of slice polyanalytic functions on quaternions. Indeed, we will use the notions of true and full slice polyanalytic Fock spaces and Segal-Bargmann transforms. We prove new properties of this QSTFT including a Moyal formula, a reconstruction formula and a Lieb's uncertainty principle. These results extend a recent paper of the authors which studies  a QSTFT having a Gaussian function as a window.
\end{abstract}
	
	
\tableofcontents
\section{Introduction}
In this decade many integral transforms have been extended to the quaternionic and Clifford algebras, see for example \cite{D,DG,KMNQ, PSS}. One of the motivations behind the study of such integral transforms in non-commutative settings is that one can deal with $n$-dimensional signals. Indeed, as observed in \cite{CK}, in image processing it is needed a higher-dimensional counterpart of the 1-dimensional signal. Moreover, the study of hypercomplex signals can be useful in other practical fields such as optics and signal processing. The reader can find more information about applications of hypercomplex signals in \cite{CK} and the references therein.
\\ One of the most studied integral transforms is the short-time Fourier transform. It is used in several applications such as the predictions of sound source position emanated by fault machine \cite{R} and the interpretation of ultrasonic waveforms \cite{NELJQ}.
The short-time Fourier transform has been studied in quaternionic and Clifford settings in \cite{BA, D1,DMD}. In particular in \cite{DMD} we gave a definition of a quaternionic short-time Fourier transform (QSTFT) in dimension one for a Gaussian window.
\\ In order to generalize our previous work, in this paper we use as window functions the weighted Hermite functions, with a parameter $ \nu=2 \pi$,
\begin{equation}
\label{Her}
\psi_n^\nu(x):= \frac{(-1)^{n} e^{\frac{\nu}{2} x^2} \displaystyle \frac{d^n}{dx^n} \bigl( e^{-\nu x^2} \bigl)}{2^{n/2} \nu^{n/2} (n!)^{1/2} \pi^{1/4} \nu^{-1/4}}.
\end{equation}
We note that for $n=0$ we have $ \psi_0(t)=2^{1/4} e^{- \pi t^2}$, which is exactly the window function that we considered in \cite{DMD}.
\\ The study of the QSTFT with respect to the weighted Hermite functions as windows is related to the theory of slice polynanalytic functions of a quaternionic variable. Recently, this topic has been intensively investigated, see \cite{ADS1, ACDS, ACDS1, ADS}.
The idea of the paper is to fix the following property
$$ V_{\psi_n} \varphi(x, \omega)= e^{- \pi i x \omega} G^{n+1} \varphi (\bar{z}) e^{- \frac{\pi | z|^2}{2}},$$
where $V_{\psi_n}$ is the complex short-time Fourier transform with respect to the weighted Hermite functions $ \psi_n$ (see \cite[Prop.1]{A}) and $G^{n+1} \varphi$ denotes the complex true polyanalytic Segal-Bargmann transform. We extend it in the quaternionic setting. To reach this aim we need the slice vision of the quaternions (\cite{GSS}). 
\\ It is possible to introduce a short-time Fourier transform of a vector-valued function $ \vec{\varphi}=(\varphi_0,..., \varphi_n)$ 
$$ \mathbf{V}_{\vec{\psi}} \vec{\varphi}(x, \omega)= e^{- \pi i x \omega} \mathbf{G} \vec{\varphi}(\bar{z}) e^{- \frac{\pi | z|^2}{2}},$$
where $ \mathbf{V}_{\vec{\psi}} \vec{\varphi}$ denotes the complex short-time Fourier transform with respect to the vector-valued window $ \vec{\psi}=( \psi_0,...,\psi_n)$ (see \cite[Formula 20]{AF}), and $\mathbf{G} \vec{\varphi}$ is the complex polyanalytic Bargmann transform (full-poly Bargmann). Also in this case we extend the formula to the quaternions.
\\ Based on the properties of the true quaternionic polyanalytic Bargmann transform and the full-poly one (see also \cite{BEA}) we prove the main results of the QSTFT. 
\\ The plan of the paper is the following: in Section 2 we recall some preliminaries about quaternions, slice hyperholomorphic functions and slice quaternionic polyanalytic theory. In Section 3 we introduce the quaternionic polyanalytic Fock space, $\widetilde{\mathcal{F}}_{Slice}^{n+1}(\mathbb{H})$, and the true one, $\mathcal{F}_T^{n}(\mathbb{H})$. Moreover, we prove the following relation between these spaces
$$ \widetilde{\mathcal{F}}_{Slice}^{n+1}(\mathbb{H})= \bigoplus_{j=0}^{n} \mathcal{F}_T^{j}(\mathbb{H}).$$
Then, we show an unitary and an isometry property of the true quaternionic polyanalytic Bargmann transform. The last property is also proved for the quaternionic full-polyanalytic Bargmann transform.
In Section 4 we prove a closed formula for the reproducing kernel of the quaternionic true polyanalytic Fock space. Using generic properties of the reproducing kernel we prove two estimates which relate the quaternionic true polyanalytic Bargmann, respectively the quaternionic full-polyanalytic Bargmann, with the signal, respectively the vector-valued signal. In Section 5 we derive a closed formula for the complex true polyanalytic Bargmann transform. To this end, we prove the following relation
$$\left(\partial_z-2\pi \overline{z}\right)^k e^{-\pi(z^2+x^2)+2\pi\sqrt{2}z x}=(-1)^k2^{-\frac{k}{2}}e^{-\pi(z^2+x^2)+2\pi\sqrt{2}zx}H_{k}\left(\frac{z+\overline{z}}{\sqrt{2}}-x\right), $$
where $H_{k}$ are the weighted Hermite polynomials. Thanks to the Identity Principle we extend the closed formula in the quaternionic setting. In Section 6 we define the true-poly QSTFT with respect to $ \psi_n$ as
$$
\mathcal{V}_{\psi_n} \varphi(x, \omega)=e^{-I \pi x \omega} B^{n+1}(\varphi) \left( \frac{\bar{q}}{\sqrt{2}} \right) e^{- \frac{|q|^2 \pi}{2}},\quad q=x+I \omega,
$$
where $B^{n+1}$ is the true quaternionic polyanalytic Bargmann transform. Furthermore, we give the definition of the QSTFT of a vector-valued signal $\vec{\varphi}=(\varphi_0,...,\varphi_n)$ with respect to the vectorial window $\vec{\psi}=(\psi_0,..., \psi_n)$, this transform will be called full-poly QSTFT,
$$ \mathbb{V}_{\vec{\psi}}\vec{\varphi}(x, \omega)=e^{-I \pi x \omega} \mathfrak{B}(\vec{\varphi}) \left( \frac{\bar{q}}{\sqrt{2}} \right) e^{- \frac{|q|^2 \pi}{2}},$$
where $\mathfrak{B}$ is the quaternionic polyanalytic Bargmann transform. These two kinds of QSTFT are related each other by the following formula
\begin{equation}
\label{intro1}
\mathbb{V}_{\vec{\psi}}\vec{\varphi}(x, \omega)= \sum_{j=0}^n \mathcal{V}_{\psi_j} \varphi_j(x, \omega).
\end{equation}
Using the properties of the true quaternionic polyanalytic Bargmann and the formula \eqref{intro1} we prove an isometric relation and a Moyal formula both for the true-poly QSTFT and the full-poly one. In this context a reconstruction formula holds both for the true-poly QSTFT and the full-poly one. For the last one we cannot have a scalar formula but a vectorial one. Basically, we use the inversion formula of the-true poly QSTFT for each component of the vector-valued function $ \vec{\varphi}=(\varphi_0,...,\varphi_n)$. The reconstruction property is also important since allows us to define the adjoint operators for these two kinds of QSTFT. These can be considered as left side inverses. Moreover, thanks to the reconstruction formula we find a formula for the reproducing kernel of the quaternionic Gabor space associated to the true-poly QSTFT, which is defined as 
$$\mathcal{G}_{\mathbb{H}}^{\psi_n}:=\lbrace{\mathcal{V}_{\psi_n}\varphi, \text{  } \varphi\in L^2(\mathbb{R},\mathbb{H})}\rbrace.$$
We also consider a vectorial-valued version of the previous space.
Finally, using the inequalities proved in section 4 we prove a Lieb's uncertainty principle, which states that (\cite{G})
\newline
\begin{center}
\emph{"A function cannot be concentrated on small sets in the time-frequency plane, no matter which  time-frequency representation is used."}
\end{center}
Due to the lack of references we add two appendices at the end of the paper. In the apppendix A, we prove an orthogonality relation for the complex Hermite polynomials, with a general parameter $ \alpha>0$. In the appendix B we show some basic properties of the Hermite polynomials, for a general parameter $ \nu>0$.

\section{Preliminaries}
The skew field of quaternions is defined to be
$$\Hq=\lbrace{q=x_0+x_1i+x_2j+x_3k\quad ; \ x_0,x_1,x_2,x_3\in\R}\rbrace$$ where the imaginary units satisfy the multiplication rules $$i^2=j^2=k^2=-1\quad \text{and}\quad ij=-ji=k, \quad jk=-kj=i,\quad ki=-ik=j.$$
On $\Hq$ the conjugate and the modulus of $q$ are defined respectively by
$$\overline{q}=Re(q)-Im(q) \quad \text{where} \quad Re(q)=x_0, \quad Im(q)=x_1i+x_2j+x_3k$$
and $$\vert{q}\vert=\sqrt{q\overline{q}}=\sqrt{x_0^2+x_1^2+x_2^2+x_3^2}.$$ We note that the quaternionic conjugation satisfy the property $\overline{ pq }= \overline{q}\, \overline{p}$ for any $p,q\in \Hq$.
Moreover, the unit sphere $$\lbrace{q=x_1i+x_2j+x_3k;\text{ } x_1^2+x_2^2+x_3^2=1}\rbrace$$ coincides with the set of all  imaginary units given by $$\mathbb{S}=\lbrace{q\in{\Hq};q^2=-1}\rbrace.$$
Any quaternion $q\in \Hq\setminus \R$ can be written in a unique way as $q=x+I y$ for some real numbers $x$ and $y>0$, and imaginary unit $I\in \mathbb{S}$, in fact we have $$q=x_0+\dfrac{x_1i+x_2j+x_3k}{|x_1i+x_2j+x_3k|}|x_1i+x_2j+x_3k|.$$
Then, for every given $I\in{\mathbb{S}}$, the slice $\C_I$ is defined to be $\mathbb{R}+\mathbb{R}I$ and it is isomorphic to the complex plane $\C$ so that it can be considered as a complex plane in $\Hq$ passing through $0$, $1$ and $I$. It is immediate that we have $$\Hq=\bigcup_{I \in \mathbb{S}}\C_I.$$ 
Before to recall the definition of slice regular functions we provide the following definition.
\begin{defn}
A domain $\Omega\subset \Hq$ is said to be a slice domain (or just $s$-domain) if  $\Omega\cap{\mathbb{R}}$ is nonempty and for all $I\in{\mathbb{S}}$, the set $\Omega_I:=\Omega\cap{\C_I}$ is a domain of the complex plane $\C_I$.
If moreover, for every $q=x+Iy\in{\Omega}$, the whole sphere $$[q]:=\lbrace{x+Jy; \, J\in{\mathbb{S}}}\rbrace,$$
is contained in $\Omega$, we say that  $\Omega$ is an axially symmetric slice domain.
\end{defn}
\begin{defn}
Let $U\subseteq\mathbb{H}$ be an axially symmetric open set and $\mathcal{U} = \{ (x,y)\in\mathbb{R}^2\ :\ x+ I y\subset U\} \subset \mathbb{R} \times \mathbb{R}$. A function $f:U\to \mathbb{H}$ is called left slice function, if it is of the form
\begin{equation}
\label{cor1}
f(q) = \alpha(x,y) + I\beta(x,y)\qquad \text{for } q = x + I y\in U
\end{equation}
with the two functions $\alpha, \beta: \mathcal{U}\to \mathbb{H}$ that satisfy the compatibility conditions
$\alpha(x,-y) = \alpha(x,y)$, $\beta(x,-y) = -\beta(x,y)$.
\end{defn}
\begin{defn}
A slice function $f: \Omega \longrightarrow \Hq$, on a given domain $\Omega\subset \Hq$, is said to be a (left) slice regular function if, for every $I\in \Sq$, the restriction $f_I$ to the slice $\C_{I}$ satisfies
$$
\overline{\partial_I} f(x+Iy):=
\dfrac{1}{2}\left(\frac{\partial }{\partial x}+I\frac{\partial }{\partial y}\right)f_I(x+Iy)=0,
$$
on $\Omega_I$. The slice derivative $\partial_S f$ of $f$ is defined by :
\begin{equation*}
\partial_S(f)(q):=
\left\{
\begin{array}{rl}
\partial_I(f)(q)& \text{if } q=x+Iy, y\neq 0\\
\displaystyle\frac{\partial}{\partial{x}}(f)(x) & \text{if } q=x \text{ is real}.
\end{array}
\right.
\end{equation*}
\end{defn}
\begin{rem}
If $f$ is a slice regular function of the form \eqref{cor1} such that $ \alpha$ and $ \beta$ are real-valued we say that $f$ is intrinsic.
\end{rem}

The space of slice regular functions is endowed with the natural topology of uniform convergence on compact sets. The characterization of slice regular functions on a ball $ B(0,R):= \{q\in \Hq; \, |q|<R\}$ centered at the origin is given by
\begin{thm}[Series expansion]
An $\Hq$-valued function $f$ is slice regular on $B(0,R)\subset \Hq$ if and only if it has a series expansion of the form:
$$f(q)=\sum_{n=0}^{+\infty} q^n\frac{1}{n!}\partial^{(n)}_S(f)(0)$$
converging on $B(0,R)=\{q\in\Hq; \,|q|<R\}$.
\end{thm}
The theory of slice regular functions has been  extensively studied in several directions, and it is nowadays widely developed \cite{ACS1,CSS1, CSS2, GSS}. The advantages of this new theory is that it contains polynomials and power series with quaternionic coefficients in the right, contrary to the Fueter theory of regular functions defined by means of the Cauchy-Riemann Fueter differential operator. The meeting point between the two function theories comes from an idea of Fueter in the thirties and next developed later by Sce \cite{CSS3, S} and by Qian \cite{Q}.
Moreover the slice theory has many applications in operator theory and in mathematical physics. The spectral theory of the S-spectrum is a natural tool for the formulation of quaternionic quantum mechanics and for the study of new classes of fractional diffusion problems, see \cite{CG, CGK},
and the references therein.
\\The  slice polyanalytic functions of a quaternionic variable (or of a paravector variable, in the case of Clifford algebra-valued functions) have to be considered as a subclass of slice functions, see \cite[Def 3.17]{ADS}.
\begin{defn}[Slice  polyanalytic functions]\label{polycorrected}
Let $n\in \mathbb{N}$ and denote by $\mathcal{C}^n(U)$ the set of
continuously differentiable functions with all their derivatives up to order $n$ on an
axially symmetric open set $U\subseteq\mathbb{H}$. A function slice $f:U\to \mathbb{H}$ is called left slice polyanalytic function of order $n\in \mathbb{N}$, if $\alpha$ and $\beta$ are in $\mathcal{C}^n(U)$ and satisfy the poly Cauchy-Riemann equations of order $n\in \mathbb{N}$
\begin{align}\label{CR}
\frac{1}{2^n}(\partial_x+I\partial_y)^n(\alpha(x,y) + I\beta(x,y))=0,\ \ \ {\rm for\ all} \ \ I\in \mathbb{S}.
\end{align}
The set of such kind of functions will be denoted by $\mathcal{SP}_{n}(\mathbb{H})$.
\end{defn}
\begin{rem}
The definition is easily adapted in the case of right slice polyanalytic functions. Moreover, we note that a slice regular function is a function as in the previous definition, when $n=1$.
\end{rem}

\begin{lem}[Splitting Lemma]\label{split1} Let $f$ be a slice polyanalytic function of order $n$ on a domain $\Omega\subseteq\Hq$. Then, for any imaginary units $I$ and $J$ with $I\perp J$ there exist $F,G:\Omega_{I}\longrightarrow{\C_I}$ polyanalytic functions of order $n$ such that for all $z=x+Iy\in\Omega_I$, we have
$$f_I(z)=F(z)+G(z)J.$$
\end{lem}

\begin{rem}
\label{split2}
For $n=1$ in Lemma \ref{split1} we obtain the classic Splitting Lemma (see \cite[Lemma 1.3]{GSS}).
\end{rem}

\begin{prop}(Poly-decomposition)
\label{Kam1}
A function $f:\Omega \to \mathbb{H}$ defined on an axially symmetric slice domain is slice polyanalytic of order $n$ if and only if there exist $f_0,...,f_{n-1}$ some unique slice regular functions on $\Omega$ such that we have the following decomposition
$$ f(q):= \sum_{k=0}^{n-1} \bar{q}^k f_k(q); \quad \forall q \in \Omega.$$
\end{prop}

\begin{prop}(Identity Principle)
\label{Kam2}
Let $f$ and $g$ be two slice polyanalytic functions of order $n$ on a slice domain $\Omega \subset \mathbb{H}$. If, for some $I \in \mathbb{S}$, $f$ and $g$ coincide on $U$ a subdomain of $\Omega_I$, then $f \equiv g$ everywhere in $\Omega$.
\end{prop}
\begin{rem}
\label{ID}
For $n=1$ we obtain the classical Identity Principle.
\end{rem}
The Leibniz rule will be useful for our calculations in the next section. The proof is based on direct computations using the definition of slice derivative . 
\begin{prop}
Let $f: \mathbb{H} \to \mathbb{H}$ be an intrinsic function and $g: \mathbb{H} \to \mathbb{H}$ be a slice regular function. Then, we have
\begin{equation}
\label{Le}
\partial_s(fg)=f (\partial_s g)+ (\partial_s f)g.
\end{equation}
\end{prop}

\begin{prop}
Let $f: \mathbb{H} \to \mathbb{H}$ be an intrinsic function and $g: \mathbb{H} \to \mathbb{H}$ be a slice regular function. Then, for any $k \in \mathbb{N}$, we have
\begin{equation}
\label{Gle}
\partial_s^k(fg)= \sum_{m=0}^k \binom{k}{m} (\partial_s^m f)(\partial_s^{k-m} g).
\end{equation}	
\end{prop}
The previous results are stated with more general hypothesis in \cite[Lemma 2.1]{EG} and \cite[Lemma 2.2]{EG}.
\\In \cite{ADS} the authors introduced the quaternionic polyanalytic Fock space defined for a given $I \in \mathbb{S}$ to be
$$ \widetilde{\mathcal{F}}_{I}^{n+1} (\mathbb{H}):= \{ f \in \mathcal{SP}_{n+1}(\mathbb{H}): \int_{\mathbb{C}_I}|f_I(q)|^2 e^{-2 \pi |q|^2} \, d \lambda_I(q) < \infty \},  \qquad n \geq 1 .$$
Moreover, the space is endowed with the following inner product
$$ \langle f, g \rangle_{ \widetilde{\mathcal{F}}_{I}^{n+1}(\mathbb{H})}= \int_{\mathbb{C}_I} \overline{g_I(q)} f_I(q)e^{-2 \pi |q|^2} \, d \lambda_I(q).$$
In \cite[Prop. 4.1]{ADS} and \cite[Prop. 4.2]{ADS} it is showed that the polyanalytic Fock space is a quaternionic reproducing kernel Hilbert space which does not depend on the choice of $ I \in \mathbb{S}$. Thus, from now we will denote the quaternionic polyanalytic Fock space by $ \widetilde{\mathcal{F}}_{Slice}^{n+1}(\mathbb{H})$.
\\Now, we give the definition of the quaternionic true polyanalytic Fock space.
\begin{defn}
\label{FT}
A function $ f: \mathbb{H} \to \mathbb{H}$ belongs to the quaternionic true polyanalytic Fock space $ \mathcal{F}_T^n (\mathbb{H})$ if and only if
\begin{itemize}
\item[i)] $ \displaystyle \int_{\mathbb{C}_I}|f_I(q)|^2 e^{-2 \pi |q|^2} \, d \lambda_I(q) < \infty.$
\item[ii)] There exists a slice regular function $H$ such that
$$ f(q)= (-1)^n\sqrt{\frac{1}{(2 \pi)^n n!}} e^{2 \pi |q|^2} \partial_s^n(e^{-2 \pi |q|^2} H(q)).$$
\end{itemize}
	
\end{defn}
\begin{rem}
We use the notation $\mathcal{F}_{Slice}(\mathbb{H})$ for the classical slice hyperholomorphic Fock space. Then, we observe that for $n=0$ we have
$$\widetilde{\mathcal{F}}_{Slice}^{1}(\mathbb{H})=\mathcal{F}_T^0 (\mathbb{H})=\mathcal{F}_{Slice}(\mathbb{H}).$$
The reproducing kernel of $\mathcal{F}_{Slice}(\mathbb{H})$ is given by 
\begin{equation}
\label{Fock}
K_{2 \pi}(p,q)=2 e_{*}(2 \pi q \bar{p}):= 2 \sum_{n=0}^{\infty} \frac{(2 \pi)^n q^n \bar{p}^n }{n!},
\end{equation}
see \cite{ACSS}.
\end{rem}
The quaternionic Segal-Bargmann transform can be defined from the quaternionic Hilbert space $L^2(\mathbb{R};dx) \! =L^2(\mathbb{R};\Hq)$, consisting of all the square integrable $\Hq$-valued functions with respect to
\begin{align}\label{spR1}
\scal{\varphi,\psi}_{L^2(\mathbb{R};dx)} : =  \int_{\R} \overline{\psi(x)} \varphi(x) dx,
\end{align}
onto the slice hyperholomorphic Bargmann-Fock space $\mathcal{F}_{Slice}(\Hq)$. For this, let $ \nu >0$, we consider the kernel function
\begin{equation}\label{KerneFct}
A(q;x) :=   \left(\frac{\nu}{\pi}\right)^{3/4} e^{\frac{-\nu}{2}(q^2+x^2)+\nu \sqrt{2}qx}; \quad (q,x)\in{\Hq\times{\mathbb{R}}},
\end{equation}
obtained as the slice hyperholomorphic extension of the kernel function of the classical Segal-Bargmann transform.
This is closely connected with the fact that $A(q;x)$ can be seen as the generating function of the real weighted Hermite functions
$$ h_n^\nu(x) := (-1)^n e^{\frac{\nu}{2}x^2} \frac{d^n}{dx^n}\left(e^{-\nu x^2}\right) $$
that form an orthogonal basis of $L^2(\mathbb{R};dx)$, with norm given explicitly by
\begin{equation}\label{normhn}
\norm{h_n^\nu}_{L^2(\mathbb{R};dx)}^2=  2^n\nu^n n!\left(\frac{\pi}{\nu}\right)^{1/2} .
\end{equation}

Associated to the kernel function $A(q;x)$ given by \eqref{KerneFct}, we consider the integral transform defined by
\begin{align} \label{defQSBT}
\mathcal{B} (\psi)(q)= \int_{\mathbb{R}}A(q;x)\psi(x)dx
= \left(\frac{\nu}{\pi}\right)^{\frac{3}{4}} \int_{\mathbb{R}}e^{\frac{-\nu}{2}(q^2+x^2)+\nu \sqrt{2}qx}\psi(x)dx
\end{align}
for $q\in{\Hq}$ and $\psi: \mathbb{R}\longrightarrow{\Hq}$, provided that the integral exists. We will call it the quaternionic Segal-Bargmann transform, see \cite{DG} for more details.

\section{Polyanalytic Bargmann transform}
Firstly, we show the relation between the quaternionic Fock space and the quaternionic true polyanalytic Fock space. In order to prove this we show two preliminary results.
\begin{lem}
\label{R1}
Let $k \geq 1$. Then, for all $ q \in \mathbb{H}$ we have
\begin{equation}
\partial_s^k e^{-2 \pi |q|^2}=(-2 \pi)^k \bar{q}^k e^{-2 \pi |q|^2}.
\end{equation}
\end{lem}
\begin{proof}
We prove the formula by induction. Let us start with $k=1$, we observe that
$$ e^{-2 \pi |q|^2}=e^{-2 \pi q \bar{q}}= \sum_{n=0}^{\infty} \frac{(-2 \pi)^n}{n!} q^n \bar{q}^n.$$
Now, we evaluate the slice derivative and get
\begin{eqnarray*}
\partial_s e^{-2 \pi |q|^2} \! \! \! \!\! \! \! \! \! \! &&= \sum_{n=1}^{\infty} \frac{(-2 \pi)^n}{n!} n q^{n-1} \bar{q}^n\\
&& = \sum_{h=0}^{\infty} \frac{(-2 \pi)^{h+1}}{(h+1)!} (h+1) q^{h} \bar{q}^{h+1}\\
&&= -2 \pi \biggl( \sum_{h=0}^{\infty} \frac{(-2 \pi)^h}{h!} q^h \bar{q}^h \biggl) \bar{q}\\
&&= -2 \pi e^{-2 \pi |q|^2} \bar{q}.
\end{eqnarray*}
Let us assume that the formula holds for $k$. We have to prove that it holds for $k+1$
\begin{eqnarray*}
\partial_s^{k+1} e^{-2 \pi |q|^2} \! \! \! \!\! \! \! \! \! \! &&= \partial_s (\partial_s^k e^{-2 \pi |q|^2} )\\
&&= (-2 \pi)^k \bar{q}^k \partial_s e^{-2 \pi |q|^2}\\
&&= (-2 \pi)^{k+1} \bar{q}^{k+1} e^{-2 \pi |q|^2}.
\end{eqnarray*}
\end{proof}
\begin{rem}
\label{rem2}
We use similar arguments to justify that for any $k \geq 1$, we have
\begin{equation}
\overline{\partial_I}^k e^{-2 \pi |q|^2}=(-2 \pi)^k q^k e^{-2 \pi |q|^2}.
\end{equation}
\end{rem}
\begin{prop}
\label{R2}
Let $g$ be a slice regular function on $ \mathbb{H}$. We consider the following function 
$$ u(q)= \sum_{k=0}^{n} (-1)^k \sqrt{\frac{1}{(2 \pi)^k k!}} e^{2 \pi |q|^2} \partial_s^k(e^{-2 \pi |q|^2} g(q)).$$
Then $ u$ is a slice polyanalytic function of order $n+1$ on $\mathbb{H}$.
\end{prop}
\begin{proof}
By the generalized Leibniz formula \eqref{Gle} we have
\begin{eqnarray}
\label{W1}
\nonumber
u(q)\! \! \! \!\! \! \! \! \! \! &&=\!\sum_{k=0}^{n}(-1)^k \sqrt{\frac{1}{(2 \pi)^k k!}} e^{2 \pi |q|^2} \partial_s^k(e^{-2 \pi |q|^2} g(q))\\ \nonumber
&&= \sum_{k=0}^{n}(-1)^k \sqrt{\frac{1}{(2 \pi)^k k!}} e^{2 \pi |q|^2} \sum_{m=0}^k \binom{k}{m} \partial_s^m e^{-2 \pi |q|^2} \partial_s^{k-m} g(q)\\ 
&&:= \sum_{k=0}^{n} c_k \underline{g}_k(q),
\end{eqnarray}
where $c_k:=(-1)^k \sqrt{\frac{1}{(2 \pi)^k k!}}$ and $\underline{g}_k(q):=e^{2 \pi |q|^2} \sum_{m=0}^k \binom{k}{m} \partial_s^m e^{-2 \pi |q|^2} \partial_s^{k-m} g(q)$.
\\ By Lemma \ref{R1} we get
\begin{eqnarray*}
\underline{g}_k(q)\! \! \! \!\! \! \! \! \! \! &&=e^{2 \pi |q|^2} \sum_{m=0}^k \binom{k}{m} (-2 \pi)^m \bar{q}^m e^{-2 \pi |q|^2} \partial_s^{k-m} g(q)\\
&&=\sum_{m=0}^k \binom{k}{m} (-2 \pi)^m \bar{q}^m  \partial_s^{k-m} g(q)\\
&&:= \sum_{m=0}^k \bar{q}^m \beta_m(q),
\end{eqnarray*}
where $\beta_m(q)= \binom{k}{m} (-2 \pi)^m \partial_s^{k-m} g(q)$.
\\Since $g$ is a slice regular function and the iteration of slice derivatives is slice regular too, we get that $\beta_m$ is slice regular. This implies by Proposition \ref{Kam1} that  $ \underline{g}_k$ is a slice polyanalytic function of order $k+1$. Finally by \eqref{W1} we get that $ u$ is a slice polyanalytic function of order $n+1$.
\end{proof}
Now, we are ready to prove the relation between the quaternionic polyanalytic Fock space and the quaternionic true polyanalytic Fock space.
\begin{thm}
\label{sum}
The quaternionic polyanalytic Fock space $ \widetilde{\mathcal{F}}_{Slice}^{n+1}(\mathbb{H})$ is the direct sum of true polynalytic Fock spaces $\mathcal{F}_T^{j}(\mathbb{H})$, $j=0,...,n$, i.e.
\label{R3}
$$ \widetilde{\mathcal{F}}_{Slice}^{n+1}(\mathbb{H})= \bigoplus_{j=0}^{n} \mathcal{F}_T^{j}(\mathbb{H}).$$
\end{thm}  
\begin{proof}
We prove the equality by double inclusion.
Let $f \in \widetilde{\mathcal{F}}_{Slice}^{n+1}(\mathbb{H})$. Let $I \in \mathbb{S}$. We choose $J \in \mathbb{S}$ be such that $ I \perp J$. Since $f$, in particular, is slice polyanalytic by the Splitting Lemma (see Lemma \ref{split1}) there exist $F,G: \mathbb{C}_I \to \mathbb{C}_I$ polyanalytic functions of order $n+1$ in  $\widetilde{\mathcal{F}}^{n+1}(\mathbb{C}_I)$ such that
$$ f_I(z)=F(z)+G(z)J.$$
By \cite{B,V} we know that
$$ \widetilde{\mathcal{F}}^{n+1}(\mathbb{C}_I)= \bigoplus_{j=0}^{n} \mathcal{F}_T^{j}(\mathbb{C}_I).$$
Therefore there exist unique $f_k, p_k \in \mathcal{F}_T^{j}(\mathbb{C}_I)$ such that
$$ F(z)= \sum_{k=0}^{n} f_k(z),$$
$$ G(z)= \sum_{k=0}^{n}p_k(z).$$
By definition of the complex true polyanalytic Fock space we have that both $f_k$ and $p_k$ satisfy the following integrability conditions
\begin{equation}
\label{gral1}
\int_{\mathbb{C}} | f_k(z)|^2 e^{-2 \pi |z|^2} dA(z) < \infty,
\end{equation}
\begin{equation}
\label{gral2}
\int_{\mathbb{C}} |p_k(z)|^2 e^{-2 \pi |z|^2} dA(z) < \infty.
\end{equation}
Moreover, they can be written as
\begin{equation}
\label{gral3}
f_k(z)= (-1)^k\sqrt{\frac{1}{(2 \pi)^k k!}} e^{2 \pi |z|^2} \partial_z^k(e^{-2 \pi |z|^2} s_k(z)),
\end{equation}
\begin{equation}
\label{gral4}
p_k(z)= (-1)^k\sqrt{\frac{1}{(2 \pi)^k k!}} e^{2 \pi |z|^2} \partial_z^k(e^{-2 \pi |z|^2} h_k(z)),
\end{equation}

where $s_k$, $h_k$ are entire functions. Thus, we have 
\begin{eqnarray}
\nonumber
f_I(z) \! \! \! \! \!&=& \! \! \! \! \! \sum_{k=0}^{n}(f_k(z)+p_k(z) J)\\ \nonumber
& =& \! \! \! \! \! \sum_{k=0}^{n}(-1)^k \sqrt{\frac{1}{(2 \pi)^k k!}} e^{2 \pi |z|^2} \partial_z^k \bigl(e^{-2 \pi |z|^2} (s_k(z)+h_k(z)J)\bigl)\\
&:=& \! \! \! \! \! \sum_{k=0}^{n} (-1)^k \sqrt{\frac{1}{(2 \pi)^k k!}} e^{2 \pi |z|^2} \partial_z^k(e^{-2 \pi |z|^2} g(z)).
\end{eqnarray}
By hypothesis we know that $f$ is  a slice polyanalytic function of order $n+1$. Moreover, by Proposition \ref{R2} we know that the following function
$$ u(q)=\sum_{k=0}^{n} (-1)^k\sqrt{\frac{1}{(2 \pi)^k k!}} e^{2 \pi |q|^2} \partial_s^k(e^{-2 \pi |q|^2} g(q))$$
is slice polyanalytic of order $n+1$. Since the functions $f$ and $u$ coincide on the slice $\mathbb{C}_I$ by the Identity Principle (see Proposition \ref{Kam2}) we have that $f(q)=u(q)$.
Now, we call
$$u_k(q):=(-1)^k\sqrt{\frac{1}{(2 \pi)^k k!}} e^{2 \pi |q|^2} \partial_s^k(e^{-2 \pi |q|^2} g(q)), \qquad 0\leq k \leq n.$$
In order to finish this first part we have to prove
$$ \int_{\mathbb{C}_I}|u_{k,I}(q)|^2e^{-2 \pi |q|^2} d \lambda_I(q) < \infty.$$
Since $g_I(z)=s_k(z)+h_k(z)J$, by \eqref{gral3} and \eqref{gral4} we have
\begin{eqnarray*}
u_{k,I}(z)&=&(-1)^k\sqrt{\frac{1}{(2 \pi)^k k!}} e^{2 \pi |z|^2} \partial_z^k(e^{-2 \pi |z|^2} s_k(z))\\
&& +(-1)^k\sqrt{\frac{1}{(2 \pi)^k k!}} e^{2 \pi |z|^2} \partial_z^k(e^{-2 \pi |z|^2} h_k(z))J\\
&=& f_k(z)+p_k(z)J.
\end{eqnarray*}
Thus by \eqref{gral1} and \eqref{gral2} we get
\begin{eqnarray*}
\int_{\mathbb{C}_I}|u_{k,I}(q)|^2e^{-2 \pi |q|^2} d \lambda_I(q) &=& \int_{\mathbb{C}_I}|f_{k,I}(z)|^2e^{-2 \pi |z|^2} d A(z)\\
&& +\int_{\mathbb{C}_I}|p_{k,I}(z)|^2e^{-2 \pi |z|^2} d A(z) < \infty.
\end{eqnarray*}

Now, we move on the other inclusion. Let $ f \in  \bigoplus_{j=0}^{n} \mathcal{F}_T^{j}(\mathbb{H})$. This means that there exist unique functions $f_k \in \mathcal{F}_T^j(\mathbb{H})$, $k=0,...,n$, such that
\begin{equation}
\label{W2}
f(q)= \sum_{k=0}^{n} f_k(q).
\end{equation}
By definition of quaternionic true polyanalytic Fock space we have that
$$ f_k(q)=(-1)^k \sqrt{\frac{1}{(2 \pi)^k k!}} e^{2 \pi |q|^2} \partial_s^k \left(e^{-2 \pi |q|^2} H(q)\right),$$
where $H$ is a slice regular function. Thus, by Proposition \ref{R2} $f$ is a slice polyanalytic function of order $n+1$. Finally, we have to prove that
$$ \int_{\mathbb{C}_I} |f_I(q)|^2 e^{-2 \pi |q|^2} \, d \lambda_I(q) < \infty.$$
By equality \eqref{W2} and the triangle inequality we have that
$$ |f_I(q)| \leq \sum_{k=0}^{n}|f_{k,I}(q)|.$$
Therefore,
$$ |f_I(q)|^2 \leq \biggl( \sum_{k=0}^{n}|f_{k,I}(q)| \biggl)^{2} \leq (n+1) \sum_{k=0}^{n}|f_{k,I}(q)|^2.$$
Now we multiply by $ e^{-2 \pi |q|^2}$ and integrate.
\begin{eqnarray*}
&& \int_{\mathbb{C}_I} |f_I(q)|^2 e^{-2 \pi |q|^2} \, d \lambda_I(q)  \leq  (n+1) \int_{\mathbb{C}_I} \sum_{k=0}^{n}|f_{k,I}(q)|^2 e^{-2 \pi |q|^2} \, d \lambda_I(q)\\
&& \leq (n+1) \biggl[ \int_{\mathbb{C}_I} |f_{1,I}(q)|^2 e^{-2 \pi |q|^2} \, d \lambda_I(q)+...+ |f_{n,I}(q)|^2 e^{-2 \pi |q|^2} \, d \lambda_I(q) \biggl] < \infty
\end{eqnarray*}
The previous conclusion holds because $f_{k} \in \mathcal{F}_T^j(\mathbb{H})$, for $k=0,...,n$, by hypothesis.
\end{proof}

\begin{rem}
A similar result was proved in \cite[Thm. 3.4]{BEA} following a different method.
\end{rem}

Now, we give the definition of the true quaternionic polyanalytic Bargmann transform for $\varphi\in L^2(\mathbb{R},\mathbb{H})$  (inspired from \cite{FLKC})
\begin{equation}
\label{new1}
B^{n+1} \varphi(q):= (-1)^n\sqrt{ \frac{1}{(2 \pi)^n n!}} \sum_{j=0}^n \binom{n}{j} (-2 \pi \bar{q})^j \partial_s^{n-j} \mathcal{B}\varphi(q),
\end{equation}
where $\mathcal{B}\varphi(q)$ is the quaternionic analogue of the Segal-Bargmann transform (see \eqref{defQSBT} with $ \nu=2 \pi$). Using the Leibniz rule and Lemma \ref{R1} we get the following definition. 
\begin{defn}\label{TPB}
The true quaternionic polyanalytic Bargmann transform of order $n+1$ of a function $\varphi \in L^{2}(\mathbb{R}, \mathbb{H})$ is defined by the formula
\begin{equation}
\label{bar1}
B^{n+1} \varphi(q)=(-1)^n\sqrt{ \frac{1}{(2 \pi)^n n!}} e^{2 \pi |q|^2} \partial_s^n [e^{-2 \pi |q|^2} \mathcal{B}\varphi(q)].
\end{equation}
\end{defn}
\begin{rem}
For $n=0$ we obtain the quaternionic Segal-Bargmann transform $B^{1}\varphi(q)=\mathcal{B}\varphi(q)$.
\end{rem}
For our future purpose  the so called quaternionic Hermite polynomials will be very important
(for more details see \cite{EG, TT}).
\begin{equation}
\label{her}
H_{m,p}^{2 \pi} (q, \bar{q}):=(-1)^{m+p} e^{2 \pi |q|^2} \partial_s^m \overline{\partial_I}^p e^{-2 \pi |q|^2}, \qquad m,p \in \mathbb{N}.
\end{equation}
\begin{rem}
Using Remark \ref{rem2} it is possible to write the quaternionic Hermite polynomials in another way:
\begin{eqnarray*}
H_{m,p}^{2 \pi}(q,\bar{q}) \! \! \! \! \! \! \! \! \!&& =(-1)^{m+p} e^{2 \pi |q|^2} \partial_s^m [(-2 \pi)^p q^p e^{-2 \pi | q|^2}]\\
&&= (-1)^{m+p} e^{2 \pi |q|^2} (-2 \pi)^p  \partial_s^m (q^p e^{-2 \pi | q|^2})\\
&&= (2 \pi)^p (-1)^{m} e^{2 \pi |q|^2}   \partial_s^m (q^p e^{-2 \pi | q|^2}).
\end{eqnarray*} 
Therefore
\begin{equation}
\label{her2}
H_{m,p}^{2 \pi}(q,\bar{q})= (2 \pi)^p (-1)^{m} e^{2 \pi |q|^2}   \partial_s^m (q^p e^{-2 \pi | q|^2}).
\end{equation}
The following orthogonality relation holds for the quaternionic Hermite polynomials (for the proof see the Appendix A (Thm. \ref{new2}))
\begin{equation}
\label{her3}
\int_{\mathbb{C}_I} \overline{ H_{m,p}^{2 \pi}(q, \bar{q})}  H_{m',p'}^{2 \pi}(q, \bar{q}) e^{-2 \pi |q|^2} \, d \lambda_I(q)=\frac{m!p!(2\pi)^{p+m}}{2}\delta_{m,m'} \delta_{p,p'}.
\end{equation}
\end{rem}
\begin{thm}
\label{iso}
The true quaternionic polyanalytic Bargmann transform $B^{n+1}: L^2(\mathbb{R},\mathbb{H}) \to \mathcal{F}^n_T(\mathbb{H})$ is an isometric isomorphism.
\end{thm}
\begin{proof}
Firstly, we remark that by Theorem \ref{R3} the norm of the true quaternionic polyanalytic Fock space $ \mathcal{F}^n_T(\mathbb{H})$ is induced by the norm of the space $\widetilde{\mathcal{F}}_{Slice}^{n+1}(\mathbb{H})$. Thus we get
$$\| B^{n+1}(\varphi) \|_{\mathcal{F}^n_T(\mathbb{H})}=\| B^{n+1}(\varphi) \|_{\widetilde{\mathcal{F}}_{Slice}^{n+1} (\mathbb{H})}.$$
Therefore, we have to prove that
\begin{equation}
\label{t1}
\| B^{n+1}(\varphi) \|_{\widetilde{\mathcal{F}}_{Slice}^{n+1} (\mathbb{H})}= \| \varphi \|_{L^2(\mathbb{R},\mathbb{H})}.
\end{equation}
Let $\varphi \in L^2(\mathbb{R},\mathbb{H})$. We expand it in the following way
$$ \varphi(x)=\sum_{k=0}^{\infty} \psi_{k}^{2 \pi}(x) \alpha_k,$$
where $\psi_{k}^{2 \pi}(x)$ are the normalized weighted Hermite functions (see \eqref{Her})
and $( \alpha_k)_{k \in \mathbb{N}} \subset \mathbb{H}$. By \cite[Lemma 4.4]{DG} we have
\begin{eqnarray}
\nonumber
\mathcal{B}\varphi(q) \! \! \! \! \! \! \! \! \! &&= \sum_{k=0}^{\infty} \mathcal{B} \left(\psi_{k}^{2 \pi}(x) \right) \alpha_k\\ \nonumber
&&= \sum_{k=0}^{\infty} \frac{2^{1/4}2^{k/2}(2 \pi)^k q^k}{2^{k/2}(2 \pi)^{k/2} \sqrt{k!} 2^{-1/4}} \alpha_k\\
&&= \sqrt{2} \sum_{k=0}^{\infty} \frac{(2 \pi)^{k/2}}{\sqrt{k!}} q^k \alpha_k.
\end{eqnarray}
Now, we insert this in \eqref{bar1} and using \eqref{her2} we obtain
\begin{eqnarray*}
B^{n+1}\varphi(q) \! \! \! \! \! \! \! \! \! &&= \sqrt{2} \sum_{k=0}^{\infty}(-1)^n \sqrt{ \frac{1}{(2 \pi)^n n!}} e^{2 \pi |q|^2} \partial_s^n \left[e^{-2 \pi |q|^2} \frac{(2 \pi)^{k/2}}{\sqrt{k!}} q^k \right] \alpha_k\\
\nonumber
&&= \sqrt{2}  \sqrt{ \frac{1}{(2 \pi)^n n!}} \sum_{k=0}^{\infty} \frac{(2 \pi)^{k/2}}{\sqrt{k!}} (-1)^n e^{2 \pi |q|^2} \partial_s^n \left[e^{-2 \pi |q|^2} q^k \right] \alpha_k\\
\nonumber
&&= \sqrt{2} \sqrt{ \frac{1}{(2 \pi)^n n!}}  \sum_{k=0}^{\infty}  \frac{1}{\sqrt{k!} (2 \pi)^{k/2}}  (-1)^n (2 \pi)^k e^{2 \pi |q|^2} \partial_s^n \left[e^{-2 \pi |q|^2} q^k \right] \alpha_k\\
&&=\sqrt{2} \sqrt{ \frac{1}{(2 \pi)^n n!}}  \sum_{k=0}^{\infty} \frac{1}{\sqrt{k!} (2 \pi)^{k/2}}  H_{n,k}^{2 \pi}(q, \bar{q}) \alpha_k.\\
\nonumber
\end{eqnarray*}
Therefore we get
\begin{equation}
\label{ant3}
B^{n+1}\varphi(q)=\sqrt{2} \sqrt{ \frac{1}{(2 \pi)^n n!}}  \sum_{k=0}^{\infty} \frac{1}{\sqrt{k!} (2 \pi)^{k/2}}  H_{n,k}^{2 \pi}(q, \bar{q}) \alpha_k.
\end{equation}
Now, we evaluate the $\widetilde{\mathcal{F}}^{n+1}_{Slice}(\mathbb{H})$ norm of $B^{n+1}$.
\begin{eqnarray*}
\nonumber
\| B^{n+1} \varphi(q) \|_{\widetilde{\mathcal{F}}^{n+1}_{Slice}(\mathbb{H})}^2\! \! \! &=& \frac{2}{n! (2 \pi)^n} \int_{\mathbb{C}_I} \biggl( \sum_{k =0}^{\infty} \frac{1}{ \sqrt{k!} (2 \pi)^{k/2}} \overline{\alpha_k} \overline{ H_{n,k}^{2\pi}(q, \bar{q})}  \biggl) \cdot \\ \nonumber
&& \cdot  \biggl( \sum_{\ell=0}^{\infty} \frac{1}{ \sqrt{\ell!} (2 \pi)^{\ell /2}}   H_{n, \ell}^{2 \pi}(q, \bar{q}) \alpha_k  \biggl) e^{-2 \pi |q|^2} \, d \lambda_I(q)\\ \nonumber
&=& \frac{2}{n! (2 \pi)^n} \sum_{k, \ell=0}^{\infty} \frac{1}{\sqrt{k!} (2 \pi)^{k/2}} \frac{1}{\sqrt{\ell!} (2 \pi)^{\ell/2}} \cdot\\
&& \cdot \overline{\alpha_k} \biggl( \int_{\mathbb{C}_I} \overline{ H_{n,k}^{2 \pi}(q, \bar{q})}  H_{n, \ell}^{2 \pi}(q, \bar{q}) e^{-2 \pi |q|^2} \, d \lambda_I(q) \biggl) \alpha_k \nonumber
\end{eqnarray*}
Due to the orthogonality relation of the quaternionic Hermite polynomials \eqref{her3} we obtain
\begin{eqnarray*}
\| B^{n+1} \varphi(q) \|_{\widetilde{\mathcal{F}}^{n+1}_{Slice}(\mathbb{H})}^2\! \! \! &=&\frac{2}{n! (2 \pi)^n} \sum_{k= 0}^{\infty} \frac{1}{k! (2 \pi)^k} \overline{\alpha_k} \biggl( \int_{\mathbb{C}_I} \overline{ H_{n,k}^{2 \pi}(q, \bar{q})}  H_{n,k}^{2 \pi}(q, \bar{q})\\
&& \cdot e^{-2 \pi |q|^2} \, d \lambda_I(q) \biggl) \alpha_k\\ 
&=& \frac{2}{n! (2 \pi)^n} \sum_{k=0}^{\infty} \frac{1}{k! (2 \pi)^k} k! n! (2 \pi)^{k+n} \frac{1}{2}  |\alpha_k|^2=\sum_{k =0}^{\infty} |\alpha_k|^2.
\end{eqnarray*}
Thus we have
\begin{equation}
\label{W4}
\| B^{n+1} \varphi(q) \|_{\widetilde{\mathcal{F}}^{n+1}_{Slice}(\mathbb{H})}^2=\sum_{k =0}^{\infty} |\alpha_k|^2.
\end{equation}
On the other hand
\begin{eqnarray*}
\nonumber
\| \varphi \|_{L^2(\mathbb{R},\mathbb{H})}^2 \! \! \! \! \! \! \! \! \! &&= \int_{\mathbb{R}} \left( \sum_{k=0}^{\infty}  \overline{ \alpha_k} \overline{\psi_{k}^{2 \pi}(x)} \right) \left( \sum_{k=0}^{\infty} \psi_{k}^{2 \pi}(x) \alpha_k \right)  \, dx\\ \nonumber
&&= \sum_{k=0}^{\infty} \overline{ \alpha_k} \left(\int_{\mathbb{R}} \overline{\psi_{k}^{2 \pi}(x)} \psi_{k}^{2 \pi}(x) \, dx \right)  \alpha_k\\ 
&&= \sum_{k=0}^{\infty} | \alpha_k|^2.
\end{eqnarray*}
Therefore we get
\begin{equation}
\label{W5}
\| \varphi \|_{L^2(\mathbb{R},\mathbb{H})}^2=\sum_{k=0}^{\infty} | \alpha_k|^2.
\end{equation}
Since \eqref{W4} and \eqref{W5} are equal we obtain \eqref{t1}. Finally, we have to prove that $B^{n+1}\varphi$ is a surjective map. This means that for a function $h \in \mathcal{F}^{n}_{T}(\mathbb{H})$ we have to find a function $ \psi \in L^{2}(\mathbb{R}, \mathbb{H})$ such that
$$ B^{n+1} \psi(q)= h(q).$$
By the definition of the quaternionic true polyanalytic Fock space $\mathcal{F}^{n}_{T}(\mathbb{H})$ (see Definition \ref{FT}) we know that there exists a slice regular function $H$ such that
$$h(q)=(-1)^n\sqrt{\frac{1}{(2 \pi)^n n!}} e^{2 \pi |q|^2} \partial_s^n(e^{-2 \pi |q|^2} H(q)).$$
From the Splitting Lemma (see Remark \ref{split2}) for slice regular functions we can write $H$ on the slice $ \mathbb{C}_I$ as
$$ H_I(z)=F(z)+G(z)J, \quad z=x+Iy \in \mathbb{C}_I,$$
where $F(z)$ and $G(z)$ are holomorphic functions.
Thus
\begin{eqnarray*}
h_I(z)\! \! \! \! \!  &=& \! \! \! \! (-1)^n\sqrt{\frac{1}{(2 \pi)^n n!}} e^{2 \pi |z|^2} \partial_z^n \left(e^{-2 \pi |z|^2} (F(z)+G(z)J) \right)\\
&=& \! \! \! \!(-1)^n\sqrt{\frac{1}{(2 \pi)^n n!}} e^{2 \pi |z|^2} \partial_z^n \left(e^{-2 \pi |z|^2} F(z)\right)\\ 
&& \! \! \! \!+ (-1)^n\sqrt{\frac{1}{(2 \pi)^n n!}} e^{2 \pi |z|^2} \partial_z^n \left(e^{-2 \pi |z|^2}G(z) \right) J\\
&:=& \! \! \! \! P(z)+Q(z)J.
\end{eqnarray*}
By hypothesis $ h \in \mathcal{F}_T^n(\mathbb{H})$, this implies that
$$ \int_{\mathbb{C}_I} |h_I(z)|^2 e^{-2 \pi |z|^2} dA(z) < \infty.$$
Therefore
\begin{eqnarray*}
\int_{\mathbb{C}_I} |P(z)|^2 e^{-2 \pi |z|^2} dA(z) &\leq&  \int_{\mathbb{C}_I} |P(z)+Q(z)J|^2 e^{-2 \pi |z|^2} dA(z)=\\
&=& \int_{\mathbb{C}_I} |h_I(z)|^2 e^{-2 \pi |z|^2} dA(z)< \infty.
\end{eqnarray*}
Hence 
$$ \int_{\mathbb{C}_I} |P(z)|^2 e^{-2 \pi |z|^2} dA(z) < \infty.$$
Using a similar reasoning we get
$$ \int_{\mathbb{C}_I} |Q(z)|^2 e^{-2 \pi |z|^2} dA(z) < \infty.$$
Moreover, since the functions $F$ and $G$ are holomorphic we obtain that $P(z)$ and $Q(z)$ belong to the space $\mathcal{F}^n_T(\mathbb{C}_I)$. Now, since the complex polyanalytic Bargmann $B^{n+1}_{\mathbb{C}}$ is an isometric isomorphism from $L^2(\mathbb{R}, \mathbb{C}_I) \to \mathcal{F}_T^n(\mathbb{C}_I)$, and in particular is surjective, we can find two functions $ \psi_1(s)$ and $\psi_2(s)$, with $s \in \mathbb{R}$, which belong to the space $L^2(\mathbb{R},\mathbb{C}_I)$ such that
$$ B^{n+1}_{\mathbb{C}_I} \psi_1(z)=P(z), \qquad B^{n+1}_{\mathbb{C}_I} \psi_2(z)=Q(z).$$
Hence
\begin{eqnarray*}
h_I(z)\! \! \! \! \! \! \! \! \! &&=P(z)+Q(z)J=B^{n+1}_{\mathbb{C}_I} \psi_1(z)+ B^{n+1}_{\mathbb{C}_I} \psi_2(z) J=\\
&&=B^{n+1}_{\mathbb{C}_I}(\psi_1(s)+\psi_2(s)J):=B^{n+1}_{\mathbb{C}_I}\psi_I(z).
\end{eqnarray*}
Finally, we get thesis by the classical Identity Principle, see Remark \ref{ID}. 
\end{proof}
\begin{lem}
\label{today}
Let $n\geq 0$ and $j,m=0,...,n$ with $j\neq m$. Then, for $f\in\mathcal{F}^j_T (\mathbb{H})$ and $g\in \mathcal{F}^m_T (\mathbb{H})$, we have 
$$\scal{f,g}_{\widetilde{\mathcal{F}}^{n+1}_{Slice} (\mathbb{H})}=0.$$
\end{lem}
\begin{proof}
Due to Theorem \ref{sum} we compute the inner product of $f$ and $g$ as the inner product of the space $\widetilde{\mathcal{F}}_{Slice}^{n+1}(\mathbb{H})$.
We note that using the definition of the true polyanalytic Fock spaces there exist two slice regular functions $H$ and $L$ such that we have 
$$f(q)=(-1)^j\sqrt{\frac{1}{(2 \pi)^j j!}} e^{2\pi|q|^2}\partial_{s}^j(e^{-2\pi|q|^2}H(q))$$
and $$g(q)=(-1)^m \sqrt{\frac{1}{(2 \pi)^m m!}} e^{2\pi|q|^2}\partial_{s}^m(e^{-2\pi|q|^2}L(q)).$$
	
Thus, we can use the series expansion theorem for slice regular functions to write 
	
$$H(q)=\displaystyle \sum_{k=0}^{\infty}q^k\alpha_k \quad \text{ and } \quad L(q)=\displaystyle \sum_{p=0}^{\infty}q^p\beta_p, \quad \lbrace{\alpha_k}\rbrace_{k\geq 0}, \lbrace{\beta_p}\rbrace_{p\geq 0}\subset \mathbb{H}.$$
	
Therefore, using the quaternionic Hermite polynomials and developing a bit the calculations we easily get that 
$$f(q)=\sqrt{\frac{1}{(2 \pi)^j j!}}\displaystyle \sum_{k=0}^{\infty} \frac{H_{k,j}^{2 \pi}(q,\bar{q})}{(2 \pi)^k}\alpha_k$$
and $$g(q)=\sqrt{\frac{1}{(2 \pi)^m m!}} \sum_{p=0}^{\infty}\frac{H_{p,m}^{2 \pi}(q,\bar{q})}{(2 \pi)^p}\beta_p.$$
Hence, using the orthogonality of the quaternionic Hermite polynomials \eqref{her3} combined with the condition $j\neq m$ we obtain 
\[ \begin{split}
\displaystyle \scal{f,g}_{\widetilde{\mathcal{F}}^{n+1}_{Slice}(\mathbb{H})}  =& \int_{\mathbb{C}_I}\overline{g(q)}f(q)e^{-2\pi|q|^2}d\lambda_I(q) \\
=&\sqrt{\frac{1}{(2 \pi)^j j!}}\sqrt{\frac{1}{(2 \pi)^m m!}} \sum_{k,p=0}^{\infty}\frac{\overline{\beta_k}}{(2 \pi)^{k+p}} \cdot \\
& \cdot \int_{\mathbb{C}_I}\left(\overline{H_{p,m}^{2 \pi}(q,\bar{q})}H_{k,j}^{2 \pi}(q,\bar{q}) e^{-2\pi|q|^2}d\lambda_I(q)\right)\alpha_p\\
 =& 0.
\end{split}
\]
\end{proof}
\begin{rem}	
We note that $\mathcal{F}^{j}_{T}(\mathbb{H})$ and $\mathcal{F}^{m}_{T}(\mathbb{H})$ are orthogonal to each other whenever $j\neq m$. This can be seen using Lemma \ref{today} and the decomposition given by Theorem \ref{sum}  $$ \widetilde{\mathcal{F}}_{Slice}^{n+1}(\mathbb{H})= \bigoplus_{j=0}^{n} \mathcal{F}_T^{j}(\mathbb{H}).$$
\end{rem}

Now, we give the proof of the following corollary for the sake of completeness.
\begin{cor}
\label{ST0}
Let $\varphi,\phi \in  L^2(\mathbb{R}, \mathbb{H})$. Then
$$ \langle B^{n+1} (\varphi), B^{n+1} (\phi) \rangle_{\mathcal{F}^n_T (\mathbb{H})}= \langle \varphi, \phi \rangle_{L^2(\mathbb{R}, \mathbb{H})}.$$
\end{cor}
\begin{proof}
It is known that any  $\varphi, \phi \in L^2(\mathbb{R}, \mathbb{H})$ can be expanded as 
$$ \varphi(x)=\sum_{k=0}^{\infty} \psi_{k}^{2 \pi}(x) \alpha_k,$$
$$ \phi(x)=\sum_{k=0}^{\infty} \psi_{k}^{2 \pi}(x) \beta_k,$$
where $(\alpha_k)_{k \geq 0}$, $(\beta_k)_{k \geq 0} \subset \mathbb{H}$. Since $ \psi_{k}^{2 \pi}(x)$ are normalized Hermite functions we have
\begin{equation}
\label{ant1}
\langle \varphi, \phi \rangle_{L^2(\mathbb{R}, \mathbb{H})}= \sum_{k=0}^{\infty} \overline{\beta_k} \alpha_k.
\end{equation}
On the other hand, by \eqref{ant3} we have
\begin{equation}
\label{U1}
B^{n+1} (\varphi)(q)=\sqrt{2} \sqrt{ \frac{1}{(2 \pi)^n n!}} \sum_{k =0}^{\infty} \frac{1}{\sqrt{k!} (2 \pi)^{k/2}}  H_{k,n}^{2 \pi}(q, \bar{q}) \alpha_k,
\end{equation}
\begin{equation}
\label{U2}
B^{n+1} (\phi)(q)=\sqrt{2} \sqrt{ \frac{1}{(2 \pi)^n n!}}  \sum_{\ell=0}^{\infty} \frac{1}{\sqrt{\ell!} (2 \pi)^{\ell/2}}  H_{\ell,n}^{2 \pi}(q, \bar{q}) \beta_k.
\end{equation}
Now, we put together \eqref{U1} and \eqref{U2} and by the orthogonality relation of the quaternionic Hermite polynomials \eqref{her3} we get
\begin{eqnarray*}
\nonumber
\langle B^{n+1} (\varphi), B^{n+1} (\phi) \rangle_{\mathcal{F}_T^n(\mathbb{H})} &=&  \! \! \! \! \! \int_{\mathbb{C}_I}\overline{B^{n+1}(\phi)(q)}B^{n+1}(\varphi)(q)e^{-2 \pi |q|^2} \, d \lambda_I(q)\\ \nonumber
&=&  \! \! \! \! \!\frac{2}{n! (2 \pi)^n} \sum_{k =0}^{\infty} \frac{1}{k! (2 \pi)^k} \overline{ \beta_k}
\biggl( \int_{\mathbb{C}_I} \overline{ H_{k,n}^{2 \pi}(q, \bar{q})}  H_{k,n}^{2 \pi}(q, \bar{q}) \\
&&  \! \! \! \! \! \cdot e^{-2 \pi |q|^2} \, d \lambda_I(q) \biggl) \alpha_k\\ 
&=&  \! \! \! \! \! \sum_{k=0}^{\infty} \overline{\beta_k} \alpha_k.
\end{eqnarray*}
Therefore
\begin{equation}
\label{ant2}
\langle B^{n+1} (\varphi), B^{n+1} (\phi) \rangle_{\mathcal{F}_T^n(\mathbb{H})}=\sum_{k=0}^{\infty} \overline{\beta_k} \alpha_k.
\end{equation}
Since \eqref{ant1} and \eqref{ant2} are equal we obtain the thesis.
\end{proof}
This notation will be very useful in the sequel. A vectorial-valued function $\vec{\varphi}=(\varphi_0,..., \varphi_n)$ is in the space 
$L^2(\mathbb{R}, \mathbb{H}^{n+1})$ if
\begin{equation}
\label{nnorm}
||\vec{\varphi}||^{2}_{L^2(\mathbb{R}, \mathbb{H}^{n+1})}:= \sum_{j=0}^{n}||\varphi_j||^{2}_{L^2(\R,\Hq)} < \infty.
\end{equation}
Moreover, it is also possible to consider an inner product for vector-valued functions $ \vec{f}=(f_0,...,f_n)$ and  $ \vec{g}=(g_0,...,g_n)$ as 
\begin{equation}
\label{inner}
\langle \vec{f}, \vec{g} \rangle_{L^{2}(\mathbb{R}^2, \mathbb{H}^{n+1})}= \sum_{j=0}^n \langle f_j, g_j \rangle_{L^{2}(\mathbb{R}^2, \mathbb{H})}.
\end{equation}
See (\cite{A2}) for more details.
\\Now, we define the quaternionic full-polyanalytic Bargmann transform.
\begin{defn}
\label{vecba}
Let $\vec{\varphi}=(\varphi_0,...,\varphi_{n})$ be a vector-valued function in $L^{2}(\mathbb{R}, \mathbb{H}^{n+1})$. The quaternionic full-polyanalytic Bargmann transform is defined as
\begin{equation}
\label{bar3}
\mathfrak{B} \vec{\varphi}(q)= \sum_{j=0}^{n} B^{j+1} \varphi_j(q),
\end{equation} 
where $B^{j+1} \varphi_j(q)$ is the true quaternionic polyanalytic Bargmann transform, defined in \eqref{bar1}. 
\end{defn}
\begin{rem}
For $n=0$ in \eqref{bar3} we obtain the quaternionic Segal-Bargmann transform.
\end{rem}
\begin{thm}
\label{cor2}
The quaternionic full-polyanalytic Bargmann transform 
$\mathfrak{B}:L^2(\mathbb{R},\Hq^{n+1}) \to \widetilde{\mathcal{F}}^{n+1}_{Slice}(\Hq)$ is an isometric isomorphism.
\end{thm}
\begin{proof}
Let $\vec{\varphi}=(\varphi_0,...,\varphi_{n})$ be a function in $L^2(\R,\Hq^{n+1})$ such that each component belongs to $L^2(\R,\Hq)$. Then, we have 
\begin{eqnarray*}
||\mathfrak{B}\vec{\varphi}(q)||^{2}_{\widetilde{\mathcal{F}}^{n+1}_{Slice}(\Hq)} \! \! \! \! \! \! \! \! \! \! &&= \int_{\mathbb{C}_I} \overline{\mathfrak{B}\vec{\varphi}(q)} \mathfrak{B}\vec{\varphi}(q) e^{-2 \pi |q|^2} \, d \lambda_I(q)=\\ 
&&=\sum_{j,m=0}^{n}\int_{\C_I}\overline{B^{j+1}(\varphi_j)(q)}B^{m+1}(\varphi_m)(q)e^{-2\pi|q|^2}d \lambda_I(q).
\end{eqnarray*}
From Lemma \ref{today} everything is zero when $j \neq m$, so we focus only on the case $j =m$. Then, by Theorem \ref{iso} we obtain 
$$\displaystyle ||\mathfrak{B}\vec{\varphi}(q)||^{2}_{\widetilde{\mathcal{F}}^{n+1}_{Slice}(\Hq)}=\sum_{j=0}^{n} \| B^{j+1} \varphi_j \|^2_{\mathcal{F}_T^n(\mathbb{H})}= \sum_{j=0}^{n}||\varphi_j||^{2}_{L^2(\R, \mathbb{H})}=||\vec{\varphi}||^{2}_{L^2(\mathbb{R}, \mathbb{H}^{n+1})}.$$	
Finally, the quaternionic full-polyanalytic Bargmann transform $\mathfrak{B}:L^2(\mathbb{R},\Hq^{n+1})\longrightarrow \widetilde{\mathcal{F}}^{n+1}_{Slice}(\Hq)$ is surjective because is the sum of true quaternionic polyanalytic Bargmann transforms, which are surjective (see Theorem \ref{iso}).
\end{proof}

\section{Reproducing kernel of the true polyanalytic Fock space}
We use the $*$-product of slice functions to express the reproducing kernel of the quaternionic true polyanalytic Fock space as follows.
We first recall this product. Given two slice functions of the form 
$f(q)=\alpha(x,y)+I\beta(x,y)$ and $g(q)=\alpha'(x,y)+I\beta'(x,y)$ with $q=x+Iy\in\mathbb{H}$ the slice product is defined by
$$(f*g)(q)=\left(\alpha\alpha'-\beta\beta'\right) +I \left(\alpha\beta'+\beta\alpha' \right), \qquad q=x+Iy\in \mathbb{H}.$$
\begin{rem}\label{prod}
If $f\in\mathcal{SP}_n(\mathbb{H})$ and $g\in\mathcal{SP}_m(\mathbb{H})$, then we have $$f*g\in \mathcal{SP}_{n+m-1}(\mathbb{H}).$$
\end{rem}
\begin{prop}\label{KernT}
The reproducing kernel of the true poly Fock space $\mathcal{F}^j_T(\mathbb{H})$, $j=1,...,n$ is given by 
$$ \! \! K_{n+1}(q,r)=2e_*(2\pi q \overline{r})*\left(\sum_{k=0}^{n} (-1)^k {n \choose{n-k}}\frac{1}{k!}(2\pi(\bar q q-q \bar r -\bar q r + \bar r r))^{k*} \right).$$
\end{prop}
\begin{proof}
Fix $r\in\Hq$ such that $r$ belongs to some slice $\C_J$, we consider the function defined by $$\displaystyle F_n^r(q)=2 e_*(2 \pi q\overline{r})*\varphi_n(q,r) $$
where for all $q\in\mathbb{H}$ we have
$$\varphi_n(q,r):=2e_*(2\pi q \overline{r})*\left(\sum_{k=0}^{n} (-1)^k {n \choose{n-k}}\frac{1}{k!}(2\pi(\bar q q-q \bar r -\bar q r + \bar r r))^{k*} \right).$$
Clearly $q\longmapsto 2e_*(2\pi q \overline{r})$ is slice regular on $\Hq$ with respect to the variable $q$. Moreover,  $\varphi_n(q,r)$ is a slice polyanalytic function of order $n+1$ on $\Hq$ with respect to $q$, by Remark \eqref{prod}. Thus, the function $F_n^r$ is a slice polyanalytic function of order $n+1$ on $\Hq$ with respect to the variable $q$. Furthermore, note that the kernel function $K_n$ extends the complex kernel on the slice $\C_J$. In particular, $F_n^r(q)$ and $K_n(q,r)$ coincide on the slice $\C_J$ containing $r$. Hence, we have $K_n(q,r)=F_n^r(q)$ everywhere on $\Hq$ thanks to the Identity Principle for slice polyanalytic functions. This ends the proof.
\end{proof}
\begin{rem}
If $n=0$ in Proposition \ref{KernT} we obtain \eqref{Fock}.
\end{rem}
Now, we are ready to show an estimate for the quaternionic-full polyanalytic Bargmann transform and the true quaternionic one.
\begin{prop}\label{estimate1}
Let $\vec{\varphi}=(\varphi_0,...,\varphi_n)$ be a vector valued function in $L^2(\mathbb{R},\mathbb{H}^{n+1})$. For every $q\in\mathbb{H}$ and every $\vec{\varphi}\in L^2(\mathbb{R},\mathbb{H}^{n+1})$, we have 
$$|\mathfrak{B} \vec{\varphi}(q)|\leq \sqrt{2(n+1)}e^{\pi |q|^2} ||\vec{\varphi}||_{L^2(\mathbb{R}, \mathbb{H}^{n+1})}.$$
\end{prop}
\begin{proof}
From \cite[Prop 4.5]{ADS} we know that for any $f\in\widetilde{\mathcal{F}}_{Slice}^{n+1}(\mathbb{H})$ and $q\in\mathbb{H}$ we have 
	
$$\displaystyle |f(q)|\leq \sqrt{2(n+1)}e^{\pi |q|^2}||f||_{\widetilde{\mathcal{F}}_{Slice}^{n+1}(\mathbb{H})}.$$
	
Now, we specify this inequality for the quaternionic full-polyanalytic Bargmann transform by setting $f(q):=\mathfrak{B}\vec{\varphi}(q)$ and get 
	
$$|\mathfrak{B}\vec{\varphi}(q)|\leq \sqrt{2(n+1)}e^{\pi|q|^2}||\mathfrak{B}\vec{\varphi}||_{\widetilde{\mathcal{F}}_{Slice}^{n+1}(\mathbb{H})}.$$
Thus, by the isometry property proved in Theorem \ref{cor2} we obtain 
$$|\mathfrak{B}\vec{\varphi}(q)|\leq \sqrt{2(n+1)}e^{\pi|q|^2}||\vec{\varphi}||_{L^2(\mathbb{R}, \mathbb{H}^{n+1})}.$$
	
\end{proof}
\begin{prop}\label{estimate2}
For every $f\in\mathcal{F}^n_T(\mathbb{H})$, we have the following estimate 
$$|f(q)|\leq \sqrt{2}e^{\pi |q|^2}||f||_{\mathcal{F}^n_T(\mathbb{H})}.$$
\end{prop}
\begin{proof}
From the reproducing kernel property of the space $\mathcal{F}_{T}^{n}(\mathbb{H})$ and the Cauchy-Schwartz inequality we have
$$|f(q)|=\left|\langle f, K_q^n \rangle_{\mathcal{F}_T^n(\mathbb{H})} \right|\leq ||f||_{\mathcal{F}^n_T(\mathbb{H})}||K^n_q||_{\mathcal{F}^n_T(\mathbb{H})}.$$
In particular, using Proposition \ref{KernT} we have $$||K^n||_{\mathcal{F}^n_T(\mathbb{H})}^2=K^n(q,q)=2 e^{2\pi|q|^2}.$$
Thus, we have $$||K^n||_{\mathcal{F}^n_T(\mathbb{H})}=\sqrt{2} e^{\pi |q|^2}.$$
	
Finally, we obtain $$|f(q)|\leq \sqrt{2}e^{\pi |q|^2} ||f||_{\mathcal{F}^n_T(\mathbb{H})}.$$
\end{proof}
\begin{prop}
\label{Kamal}
For every $q\in\mathbb{H}$ and $\varphi\in L^2(\mathbb{R},\mathbb{H})$, we have $$|B^{n+1}\varphi(q)|\leq \sqrt{2}e^{\pi|q|^2}||\varphi ||_{L^2(\mathbb{R},\mathbb{H})}.$$
\end{prop}
\begin{proof}
We follow a similar reasoning of Proposition \ref{estimate1}, then we apply Theorem 3.9.
\end{proof}
\begin{rem}
For $n=0$ in Proposition \ref{estimate1} and Proposition \ref{Kamal} we get the same estimate of \cite[Prop. 4.3]{DG} with $\nu=2\pi$.
\end{rem}

\section{Closed formula of the true quaternionic polyanalytic Bargmann transform  }
Inspired from Vasilevski paper \cite{V} we write the following closed integral transform of the true polyanalytic Bargmann in the complex case, for $ \varphi \in L^2(\mathbb{C})$
\begin{equation}
\label{Kamal1}
\widehat{B}^{k+1}\varphi (z)=2^{\frac{3}{4}}(2^kk!(2\pi)^k)^{-\frac{1}{2}}\int_\mathbb{R}e^{-\pi(z^2+x^2)+2\pi\sqrt{2}zx}H_{k}\left(\frac{z+\bar{z}}{\sqrt{2}}-x\right)\varphi(x)dx,
\end{equation}
where $H_k$ are the weighted Hermite polynomials defined as
\begin{equation}
\label{herm2}
H_j(y)=(-1)^j e^{2 \pi y^2} \frac{d^j}{dy^j} e^{-2 \pi y^2}= j! \sum_{m=0}^{[\frac{j}{2}]} \frac{(-1)^m (4 \pi y)^{j-2m}}{m! (j-2 m)!},
\end{equation}
and $[.]$ denotes the integer part.
We want to prove the equality between $\widehat{B}^{k+1}\varphi (z)$ and $B^{k+1}\varphi (z)$, which is defined by
$$
B^{k+1}\varphi(z)=\sqrt{\frac{1}{(2\pi)^k k!}}(-1)^k\displaystyle \sum_{j=0}^k\binom{k}{j}(-2\pi\overline{z})^j \partial_{z}^{k-j}(B\varphi)(z).
$$
Thus, by definition of the Segal-Bargmann transform we have 
\begin{eqnarray*}
B^{k+1}\varphi(z) \! \! \! \! &=& \! \! \!  \!2 ^{ \frac{3}{4}}\sqrt{\frac{1}{(2\pi)^k k!}}(-1)^k\displaystyle\int_\mathbb{R} \sum_{j=0}^k\binom{k}{j}(-2\pi\overline{z})^j \partial_{z}^{k-j}\left(e^{-\pi(z^2+x^2)+2\pi\sqrt{2}zx}\right)\\
&& \cdot \varphi(x)dx\\
&=& \! \! \!  \! 2 ^{ \frac{3}{4}}\sqrt{\frac{1}{(2\pi)^k k!}}(-1)^k\displaystyle\int_\mathbb{R} \left[\left(\partial_z-2\pi \overline{z}\right)^k e^{-\pi(z^2+x^2)+2\pi\sqrt{2}z x} \right] \varphi(x) dx.
\end{eqnarray*}
Then, in order to prove the equality between $\widehat{B}^{k+1}$ and $B^{k+1}$ we need the following result.
\begin{prop}\label{Bhat}
For any $k\geq 0$, $z \in \mathbb{C}$ and $x \in \mathbb{R}$, we have 
$$\left(\partial_z-2\pi \overline{z}\right)^k e^{-\pi(z^2+x^2)+2\pi\sqrt{2}z x}=(-1)^k2^{-\frac{k}{2}}e^{-\pi(z^2+x^2)+2\pi\sqrt{2}zx}H_{k}\left(\frac{z+\overline{z}}{\sqrt{2}}-x\right).$$
\end{prop}
\begin{proof}	
We prove the statement by induction. For $k=0$ we have $H_0\left(\frac{z+\overline{z}}{\sqrt{2}}-x\right)=1$, thus the result holds in this case. Let us assume that the equality is true for $k$. We will prove the result for $k+1$ thanks to the induction hypothesis and Leibniz rule. Indeed, we have 

\begin{eqnarray*}
&& \! \! \! \!  \! \!  (\partial_z-2\pi\overline{z})^{k+1} e^{-\pi(z^2+x^2)+2\pi\sqrt{2}zx}\\ 
&=& \! \! \! \! (\partial_z-2\pi\overline{z})(\partial_z-2\pi\overline{z})^k \left(e^{-\pi(z^2+x^2)+2\pi\sqrt{2}zx} \right) \\
&=& \! \! \! \! (\partial_z-2\pi\overline{z})\left[(-1)^k2^{-\frac{k}{2}}e^{-\pi(z^2+x^2)+2\pi\sqrt{2}zx}H_{k}\left(\frac{z+\overline{z}}{\sqrt{2}}-x\right)\right] \\
&=& \! \! \! \! (-1)^k2^{-\frac{k}{2}} \biggl[ \partial_z(e^{-\pi(z^2+x^2)+2\pi\sqrt{2}zx})H_{k}\left(\frac{z+\overline{z}}{\sqrt{2}}-x\right) \\
&&  \! \! \! \!+ e^{-\pi(z^2+x^2)+2\pi\sqrt{2}zx}\partial_z  H_{k}\left(\frac{z+\overline{z}}{\sqrt{2}}-x\right) \\
&&\! \! \! \! -2\pi\overline{z} H_{k}\left(\frac{z+\overline{z}}{\sqrt{2}}-x\right)  e^{-\pi(z^2+x^2)+2\pi\sqrt{2}zx} \biggl]. 
\end{eqnarray*}	

We write $z=u+iv$ and develop the computations using formula \eqref{A2} in Appendix B (with $\nu=2 \pi$) to get
\[ \begin{split}
\displaystyle \partial_z  H_{k}\left(\frac{z+\overline{z}}{\sqrt{2}}-x\right) & =\frac{1}{2}\left(\frac{\partial}{\partial u}-i\frac{\partial}{\partial v}\right)H_k(\sqrt{2}u-x)  \\
&=\frac{\sqrt{2}}{2} \frac{d}{du}H_k(\sqrt{2}u-x) \\
&= 2\sqrt{2}\pi kH_{k-1}\left(\frac{z+\overline{z}}{\sqrt{2}}-x\right)
\end{split}
\]
and  $$\partial_z\left( e^{-\pi(z^2+x^2)+2\pi\sqrt{2}zx} \right)=(-2\pi z+2\pi\sqrt{2}x) e^{-\pi(z^2+x^2)+2\pi\sqrt{2}zx}. $$
	
Thus, if we set $\displaystyle y=\frac{z+\overline{z}}{\sqrt{2}}-x$ and by using formula \eqref{A1} in Appendix B (with $\nu=2 \pi$) we obtain
\begin{eqnarray*}
&& \! \! \! \! (\partial_z-2\pi\overline{z})^{k+1} e^{-\pi(z^2+x^2)+2\pi\sqrt{2}zx}\\
&=& \! \! \! \! (-1)^k2^{-\frac{k}{2}} e^{-\pi(z^2+x^2)+2\pi\sqrt{2}zx} \bigl[(-2\pi z+2\pi\sqrt{2}x)H_{k}(y)\\  
&& \! \! \! \!+ 2\sqrt{2}\pi kH_{k-1}(y)-2\pi\overline{z} H_{k}(y) \biggl] \\
&=& \! \! \! \! (-1)^{k+1}\frac{2^{-\frac{k}{2}}}{\sqrt{2}} \sqrt{2} e^{-\pi(z^2+x^2)+2\pi\sqrt{2}zx}\left[ 2\pi\sqrt{2}yH_k(y) -2\sqrt{2}\pi k H_{k-1}(y) \right] \\
&=& \! \! \! \! (-1)^{k+1}2^{-\frac{k+1}{2}}e^{-\pi(z^2+x^2)+2\pi\sqrt{2}zx}\left(4\pi yH_k(y)-4\pi k  H_{k-1}(y)\right)\\ 
&=& \! \! \! \! (-1)^{k+1}2^{-\frac{k+1}{2}}e^{-\pi(z^2+x^2)+2\pi\sqrt{2}zx}H_{k+1}(y).
\end{eqnarray*}

Thus replacing $y$ by $\frac{z+\overline{z}}{\sqrt{2}}-x$ we have the result for $k+1$. 
\end{proof}
Due to Proposition \ref{Bhat} we have that 
$$\widehat{B}^{k+1}\varphi(z)=B^{k+1}\varphi(z).$$
\begin{lem}
The weighted Hermite polynomials $H_k\left(\frac{q+\overline{q}}{\sqrt{2}}-x\right)$, with $ q \in \mathbb{H}$ and $x \in \mathbb{R}$, are slice polyanalytic of order $k+1$ on $ \mathbb{H}$.
\end{lem}
\begin{proof}
We know that $$\displaystyle H_k\left(\frac{q+\overline{q}}{\sqrt{2}}-x\right)=k!\sum_{m=0}^{[\frac{k}{2}]}\frac{(-1)^m \left(4\pi(\frac{q+\overline{q}}{\sqrt{2}}-x) \right)^{k-2m}}{m!(k-2m)!}.$$
We note that $H_k\left(\frac{q+\overline{q}}{\sqrt{2}}-x \right)$ is a slice function since it is the sum of slice functions. To justify that it is slice polyanalytic of order $k+1$, we proceed by induction on $k$. In order, to get the thesis it is enough to prove that 
$$\overline{\partial_I}^{k+1}\left(4\pi\left(\frac{q+\overline{q}}{\sqrt{2}}-x\right)^{k-2m}\right)=0, \qquad  0 \leq m\leq \left[\frac{k}{2}\right].$$
Let us begin the induction: the case $k=1$ is trivial. Now, we assume that the statement holds for $k$ and we prove it for $k+1$. We have by induction hypothesis 

\begin{eqnarray*}
\overline{\partial_I}^{k+2}\left(4\pi\left(\frac{q+\overline{q}}{\sqrt{2}}-x \right)\right)^{k+1-2m}
\! \! \! \! &=& \! \! \! \! \overline{\partial_I}^{k+1}\overline{\partial_I}\left(4\pi\left(\frac{q+\overline{q}}{\sqrt{2}}-x \right)\right)^{k+1-2m}\\
&=& \! \! \! \!(k+1-2m)2\sqrt{2}\pi \\
&& \! \! \! \! \cdot\overline{\partial_I}^{k+1}\left(4\pi\left(\frac{q+\overline{q}}{\sqrt{2}}-x \right)\right)^{k-2m} \\
&=& \! \! \! \! 0.
\end{eqnarray*}

This means that when we apply $\overline{\partial_I}^{k+1}$ to $H_{k}\left(\frac{q+\overline{q}}{\sqrt{2}}-x\right)$ each addendum is zero. Thus, $$\overline{\partial_I}^{k+1}H_k\left(\frac{q+\overline{q}}{\sqrt{2}}-x\right)=0.$$
Let us recall $\widehat{B}^{k+1}\varphi$ expression in the quaternionic setting 
$$\displaystyle \widehat{B}^{k+1}\varphi(q)=2^{\frac{3}{4}}(2^k k!(2\pi)^{k})^{-\frac{1}{2}}\int_{\mathbb{R}}e^{-\pi(q^2+x^2)+2\pi\sqrt{2}qx}H_{k}\left(\frac{q+\overline{q}}{\sqrt{2}}-x\right)\varphi(x) dx.$$
	
Since the function $G(q):=e^{-\pi(q^2+x^2)+2\pi\sqrt{2}x}$ is slice regular and $H_k\left(\frac{q+\bar{q}}{\sqrt{2}}-x\right)$ is intrinsic and we proved that it is slice polyanalytic of order $k+1$. Then, by \cite[Prop 3.3]{ADS} we have that the function $e^{-\pi(q^2+x^2)+2\pi\sqrt{2}q}H_{k}\left(\frac{q+\overline{q}}{\sqrt{2}}-x\right)$ is slice polyanalytic of order $k+1$. This means that $\widehat{B}^{k+1}\varphi$ is slice polyanalytic of order $k+1$.
\end{proof}

\begin{prop}
The two true quaternionic full-polyanalytic Bargmann transforms $B^{k+1}$ and $\widehat{B}^{k+1}$ are equal.
\end{prop}
\begin{proof}
Since $B^{k+1}\varphi(z)=\widehat{B}^{k+1}\varphi(z)$ and $\widehat{B}^{k+1}\varphi, B^{k+1}\varphi$ are slice polyanalytic of order $k+1$ by the Identity Principle for slice polyanalytic functions we get that 
$$\widehat{B}^{k+1}\varphi(q)= B^{k+1}\varphi(q).$$
\end{proof}
\begin{rem}
The formula for $\widehat{B}^{k+1}\varphi$ is a closed formula for the polyanalytic Bargmann transform. Moreover, for $k=0$ it turns out that $\widehat{B}^{k+1}\varphi$ reduces to the quaternionic analogue of the Segal-Bargmann transform (see \eqref{defQSBT} with $\nu=2 \pi$).
\end{rem}

\section{Quaternion short-time Fourier transform with normalized Hermite functions as windows }
The short-time Fourier transform provides a simultaneous description of the temporal and spectral behaviour of a signal, which varies over the time. In order to find the frequency spectrum of a signal $ \varphi$ at a specific time $x$, one can localize the signal $ \varphi$ to neighbourhood of $x$ and after evaluates the Fourier transform of the restriction. This procedure of localization is made by choosing  a cut-off function, called "\emph{window function}".
\\ The aim of this section is to introduce a quaternionic analogue of the short-time Fourier transform in dimension one with normalized weighted Hermite functions as windows, $\psi_n(t)= \frac{h_n^{2 \pi}(t)}{\| h_n^{2 \pi} \|_{L^2(\mathbb{R}, \mathbb{H})}}$, see \eqref{Her}. To develop this concept we need the theory of slice polyanalytic functions, see \cite{ADS}. We start by considering this formula \cite[Prop.1]{A}
\begin{equation}
\label{start}
V_{\psi_n} \varphi(x, \omega)= e^{- \pi i x \omega} G^{n+1} \varphi \left (\frac{\bar{z}}{\sqrt{2}} \right) e^{- \frac{\pi | z|^2}{2}},
\end{equation}
where the variables $(x, \omega) \in \mathbb{R}^2$ have been converted into a complex vector $z=x+i \omega$, and $G^{n+1} \varphi(z)$ is the complex true polyanalytic version of the Segal-Bargmann transform.
\\ In this context it is possible to consider a quaternion short-time Fourier transform of a vector-valued function $ \vec{\varphi}=(\varphi_0,..., \varphi_n)$ with respect to $ \vec{\psi}=( \psi_0,...,\psi_n)$. Also for this kind of signal it is possible to have a relation as \eqref{start}. Let us consider the following formula \cite[Formula 20]{AF}
\begin{equation}
\label{start1}
\mathbf{V}_{\vec{\psi}} \vec{\varphi}(x, \omega)= e^{- \pi i x \omega} \mathbf{G} \vec{\varphi} \left (\frac{\bar{z}}{\sqrt{2}} \right) e^{- \frac{\pi | z|^2}{2}},
\end{equation}
where $\mathbf{G}$ is the complex full-polyanalytic Segal-Bargmann transform. We want to extend \eqref{start} and \eqref{start1} to the quaternionic setting. 
\begin{defn}
\label{ST1}
Let $ \varphi: \mathbb{R} \to \mathbb{H}$ be a function in $L^2(\mathbb{R}, \mathbb{H})$. We define the 1D-true polyanalytic quaternion short time Fourier transform (true-poly QSTFT) with respect to $ \psi_n(t)= \frac{h_n^{2 \pi}(t)}{\| h_n^{2 \pi} \|_{L^2(\mathbb{R}, \mathbb{H})}}$ as
\begin{equation}
\label{QST1}
\mathcal{V}_{\psi_n} \varphi(x, \omega)=e^{-I \pi x \omega} B^{n+1}(\varphi) \left( \frac{\bar{q}}{\sqrt{2}} \right) e^{- \frac{|q|^2 \pi}{2}},
\end{equation}
where $q=x+I \omega$ and $B^{n+1}$ is the true quaternionic polyanalytic Bargmann transform, defined in \eqref{bar1}.
\end{defn}
It is possible to define a vector-valued quaternionic short-time Fourier transform.
\begin{defn}
Let $ \vec{\varphi}=( \varphi_0,..., \varphi_n)$ be a vector-valued function in $L^2(\mathbb{R}, \mathbb{H}^{n+1})$. We define the 1D-polyanalytic quaternion short-time Fourier transform (full-poly QSTFT) with respect to $ \vec{\psi}= (\psi_0,..., \psi_n)$ as
\begin{equation}
\label{QST2}
\mathbb{V}_{\vec{\psi}}\vec{\varphi}(x, \omega)=e^{-I \pi x \omega} \mathfrak{B}(\vec{\varphi}) \left( \frac{\bar{q}}{\sqrt{2}} \right) e^{- \frac{|q|^2 \pi}{2}},
\end{equation}
where $q=x+I \omega$ and $ \mathfrak{B}$ is the quaternionic full-polyanalytic Bargmann transform, defined in \eqref{bar3}.
\end{defn}
\begin{rem}
For $n=0$ in \eqref{QST1} and \eqref{QST2} we obtain the definition of the 1D-quaternion short-time Fourier transform with respect to the Gaussian window $ g(t)=2^{1/4} e^{- \pi t^2}$, see \cite[Def. 5.1]{DMD}. Indeed, by \eqref{Her} we get
$$ \psi_0(t)= \frac{h_0^{2 \pi}(t)}{ \|h_0^{2 \pi} \|_{L^2(\mathbb{R}, \mathbb{H})}}=\frac{e^{- \pi t^2}}{\left( \frac{\pi}{2 \pi} \right)^{\frac{1}{4}}}=2^{1/4} e^{- \pi t^2}.$$
Moreover, we have already observed that $B^1 \varphi= \mathcal{B} \varphi$, which is the quaternionic analogue of the Bargmann transform. Therefore, with formulas \eqref{QST1} and \eqref{QST2} we are working in a more general setting with respect to the paper \cite{DMD}.
\end{rem}

It is possible to put in relation the true-poly QSTFT and the full-poly one. Indeed, we have the following result.
\begin{prop}
\label{sum2}
Let $ \vec{\varphi}=( \varphi_0,..., \varphi_n)$ be a vector valued function in $L^2( \mathbb{R}, \mathbb{H}^{n+1})$. The sum of the true-poly QSTFTs with respect to $ \psi_j$ of $ \varphi_j$, with $0 \leq j \leq n$, is the full-poly QSTFT with respect to $ \vec{\psi}$ of $ \vec{\varphi}$, i.e.
\begin{equation}
\label{sum1}
\mathbb{V}_{\vec{\psi}}\vec{\varphi}(x, \omega)= \sum_{j=0}^n \mathcal{V}_{\psi_j} \varphi_j(x, \omega).
\end{equation}
\end{prop}
\begin{proof}
From \eqref{QST2} and \eqref{bar3} turns out that
\begin{eqnarray*}
\mathbb{V}_{\vec{\psi}}\vec{\varphi}(x, \omega) &=&e^{-I \pi x \omega} \mathfrak{B}(\vec{\varphi}) \left( \frac{\bar{q}}{\sqrt{2}} \right) e^{- \frac{|q|^2 \pi}{2}}\\
&=& e^{-I \pi x \omega} \left( \sum_{j=0}^n B^{j+1} \varphi_j \left( \frac{\bar{q}}{\sqrt{2}} \right) \right)e^{- \frac{|q|^2 \pi}{2}}\\
&=&\sum_{j=0}^n e^{-I \pi x \omega}   B^{j+1} \varphi_j \left( \frac{\bar{q}}{\sqrt{2}} \right) e^{- \frac{|q|^2 \pi}{2}}.
\end{eqnarray*}
Thus, by \eqref{QST1} we have
$$\mathbb{V}_{\vec{\psi}} \vec{\varphi}(x, \omega)= \sum_{j=0}^n \mathcal{V}_{\psi_j} \varphi_j(x, \omega).$$
\end{proof}
\subsection{Moyal formulas}
Here, we show that a Moyal formula and an isometric relation hold both for the true-poly QSTFT and the full-poly one.
\begin{thm}
\label{Moy1}
Let $ \varphi, \phi$ be functions in $L^{2}(\mathbb{R}, \mathbb{H})$. Then we have
\begin{equation}
\label{iso1}
\langle \mathcal{V}_{\psi_n} \varphi, \mathcal{V}_{\psi_n} \phi \rangle_{L^2(\mathbb{R}^2, \mathbb{H})} =2 \langle\varphi,  \phi \rangle_{L^2( \mathbb{R}, \mathbb{H})}.
\end{equation}
\end{thm}
\begin{proof}
By Definition \ref{ST1} we have
\begin{eqnarray*}
\langle \mathcal{V}_{\psi_n}\varphi , \mathcal{V}_{\psi_n} \phi \rangle_{L^2(\mathbb{R}^2, \mathbb{H})} &=& \int_{\mathbb{R}^2} \overline{\mathcal{V}_{\psi_n} \phi (x, \omega)} \mathcal{V}_{\psi_n} \varphi(x, \omega)  d xd \omega\\
&=& \int_{\mathbb{R}^2} \overline{e^{-I \pi x \omega} B^{n+1}(\phi) \left( \frac{\bar{q}}{\sqrt{2}} \right) e^{- \frac{|q|^2 \pi}{2}}} \cdot\\ 
&& \cdot e^{-I \pi x \omega} B^{n+1}(\varphi) \left( \frac{\bar{q}}{\sqrt{2}} \right) e^{- \frac{|q|^2 \pi}{2}}  dxd \omega\\
&=& \! \!\int_{\mathbb{R}^2} \overline{B^{n+1}(\phi) \left( \frac{\bar{q}}{\sqrt{2}} \right)}e^{I \pi x \omega}  e^{-I \pi x \omega} B^{n+1}(\varphi) \left( \frac{\bar{q}}{\sqrt{2}} \right) \\
&& \cdot e^{- |q|^2 \pi}  dxd \omega\\
&=& \! \!\int_{\mathbb{R}^2} \overline{B^{n+1}(\phi) \left( \frac{\bar{q}}{\sqrt{2}} \right)} B^{n+1}(\varphi) \left( \frac{\bar{q}}{\sqrt{2}} \right) e^{- |q|^2 \pi} dx d \omega
\end{eqnarray*}
We put $ p= \frac{\bar{q}}{\sqrt{2}}$ and by Corollary \ref{ST0} we get
\begin{eqnarray*}
\langle \mathcal{V}_{\psi_n} \varphi, \mathcal{V}_{\psi_n} \phi \rangle_{L^2(\mathbb{R}^2, \mathbb{H})} &=&  2\int_{\mathbb{R}^2} \overline{B^{n+1} \left(\phi(p)\right) } B^{n+1}\left(\varphi(p)\right) e^{- 2|p|^2 \pi} dp\\
&=&2  \langle B^{n+1} \varphi, B^{n+1} \phi \rangle_{\mathcal{F}^n_T( \mathbb{H})}.\\
&=& 2 \langle \varphi,\phi \rangle_{L^2( \mathbb{R}, \mathbb{H})}.
\end{eqnarray*}
\end{proof}
\begin{cor}
\label{iso2}
For any $ \sigma \in L^2( \mathbb{R}, \mathbb{H})$
$$ \| \mathcal{V}_{\psi_n} \sigma \|_{L^2( \mathbb{R}^2, \mathbb{H})}= \sqrt{2} \| \sigma \|_{L^2( \mathbb{R}, \mathbb{H})}.$$
\end{cor}
\begin{proof}
It follows trivially by Theorem \ref{Moy1} putting $ \phi= \varphi:= \sigma.$
\end{proof}
We prove an isometry property for the full-poly QSTFT.
\begin{thm}
\label{iso5}
Let $ \vec{\varphi}=(\varphi_0,..., \varphi_n)$ be a vector valued function in $L^2( \mathbb{R}, \mathbb{H}^{n+1})$, then
$$ \| \mathbb{V}_{\vec{\psi}} \vec{\varphi} \|_{L^2( \mathbb{R}^2, \mathbb{H})}= \sqrt{2} \| \vec{\varphi} \|_{L^2( \mathbb{R}, \mathbb{H}^{n+1})}.$$
\end{thm}
\begin{proof}
First of all we note that
$$ \|\mathbb{V}_{\vec{\psi}} \vec{\varphi} \|_{L^2( \mathbb{R}^2, \mathbb{H})}^2= \sum_{j=0}^n \| \mathcal{V}_{\psi_j} \varphi_j \|^2_{L^2(\mathbb{R}^2, \mathbb{H})},$$
where $ \mathcal{V}_{\psi_j}$ are the true-poly QSTFTs. Therefore by Corollary \ref{iso2} we have
\begin{eqnarray*}
\|\mathbb{V}_{\vec{\psi}_j} \vec{\varphi} \|_{L^2( \mathbb{R}^2, \mathbb{H})}^2 &=&\| \mathcal{V}_{\psi_0} \varphi_0 \|^2_{L^2(\mathbb{R}^2, \mathbb{H})}+...+\| \mathcal{V}_{\psi_n} \varphi_n \|^2_{L^2(\mathbb{R}^2, \mathbb{H})}\\
&=& 2 \|  \varphi_0 \|^2_{L^2(\mathbb{R}, \mathbb{H})}+...+ 2 \|  \varphi_n \|^2_{L^2(\mathbb{R}, \mathbb{H})}\\
&=& 2 \sum_{j=0}^n \| \varphi_j \|^2_{L^2(\mathbb{R}, \mathbb{H})}=2\| \vec{\varphi} \|_{L^2( \mathbb{R}, \mathbb{H}^{n+1})}^2.
\end{eqnarray*}
\end{proof}
Now, we prove a Moyal formula for the full-poly QSTFT. In order to do this we need the following polarization identity (see \cite[Formula 2.4]{GMP}) for $u, v \in \mathbb{H}$
$$ \langle u, v \rangle= \frac{1}{4} \left( \| u+v \|^2- \|u-v\|^2 \right)+ \frac{1}{4} \sum_{\boldsymbol{\tau}=\boldsymbol{i},\boldsymbol{j},\boldsymbol{k}}\left( \| u\boldsymbol{\tau}+v \|^2- \|u\boldsymbol{\tau}-v\|^2 \right)\boldsymbol{\tau}.$$
\begin{prop}
Let $ \vec{\varphi}=( \varphi_0,..., \varphi_n)$ and  $ \vec{\phi}=(\phi_0,..., \phi_n)$ be vector valued functions in $L^{2}( \mathbb{R}, \mathbb{H}^{n+1})$. Then, we have
$$ \langle \mathbb{V}_{\vec{\psi}} \vec{\phi}, \mathbb{V}_{\vec{\psi}} \vec{\varphi} \rangle_{{L^{2}(\mathbb{R}^2, \mathbb{H})}}=2 \langle \vec{\phi}, \vec{\varphi} \rangle_{L^{2}(\mathbb{R}, \mathbb{H}^{n+1})}.$$
\end{prop}
\begin{proof}
By the polarization identity with $u:=\mathbb{V}_{\vec{\psi}} \vec{\phi}$ and $v:=\mathbb{V}_{\vec{\psi}} \vec{\varphi}$ and the linearity of the full-poly QSTFT, which comes from the linearity of the quaternionic full-polyanalytic Bargmann, we get
\begin{eqnarray*}
\langle \mathbb{V}_{\vec{\psi}} \vec{\phi}, \mathbb{V}_{\vec{\psi}} \vec{\varphi} \rangle_{{L^{2}(\mathbb{R}^2, \mathbb{H})}}&=& \frac{1}{4} \left( \| \mathbb{V}_{\vec{\psi}} \vec{\phi}+\mathbb{V}_{\vec{\psi}} \vec{\varphi}  \|^2- \|\mathbb{V}_{\vec{\psi}} \vec{\phi}-\mathbb{V}_{\vec{\psi}} \vec{\varphi}\|^2 \right)+ \\
&& +\frac{1}{4} \sum_{\boldsymbol{\tau}=\boldsymbol{i},\boldsymbol{j},\boldsymbol{k}}\left( \| \mathbb{V}_{\vec{\psi}} \vec{\phi} \cdot \boldsymbol{\tau}+\mathbb{V}_{\vec{\psi}} \vec{\varphi}  \|^2- \|\mathbb{V}_{\vec{\psi}} \vec{\phi} \cdot \boldsymbol{\tau}-\mathbb{V}_{\vec{\psi}} \vec{\varphi} \|^2 \right)\boldsymbol{\tau}\\
&=& \frac{1}{4} \left( \| \mathbb{V}_{\vec{\psi}} (\vec{\phi}+\vec{\varphi})  \|^2- \|\mathbb{V}_{\vec{\psi}} (\vec{\phi}-\vec{\varphi})\|^2 \right)+ \\
&& +\frac{1}{4} \sum_{\boldsymbol{\tau}=\boldsymbol{i},\boldsymbol{j},\boldsymbol{k}}\left( \| \mathbb{V}_{\vec{\psi}} (\vec{\phi} \cdot \boldsymbol{\tau}+\vec{\varphi} ) \|^2- \|\mathbb{V}_{\vec{\psi}} (\vec{\phi} \cdot \boldsymbol{\tau}- \vec{\varphi} )\|^2 \right)\boldsymbol{\tau}.
\end{eqnarray*}
Since the space $L^2(\mathbb{R}, \mathbb{H}^{n+1})$ is a right vector quaternionic space we have that $ \vec{\phi} \pm \vec{\varphi}$ and $\vec{\phi}  \boldsymbol{\tau}\pm\vec{\varphi}$ stay in $L^{2}(\mathbb{R}, \mathbb{H}^{n+1})$. Therefore by Theorem \ref{iso5} we have
\begin{eqnarray}
\label{iso4}
\nonumber
\langle \mathbb{V}_{\vec{\psi}} \vec{\phi}, \mathbb{V}_{\vec{\psi}} \vec{\varphi} \rangle_{{L^{2}(\mathbb{R}^2, \mathbb{H})}}&=& \frac{1}{4}2 \left( \| \vec{\phi}+\vec{\varphi} \|^2- \|\vec{\phi}-\vec{\varphi}\|^2 \right)+ \\ 
&& +\frac{1}{4} 2\sum_{\boldsymbol{\tau}=\boldsymbol{i},\boldsymbol{j},\boldsymbol{k}}\left( \| \vec{\phi} \cdot \boldsymbol{\tau}+\vec{\varphi} \|^2- \|\vec{\phi} \cdot \boldsymbol{\tau}- \vec{\varphi} \|^2 \right)\boldsymbol{\tau}.
\end{eqnarray}
Using another time  the polarization identity with $ u:= \vec{\phi}$ and $v:= \vec{\varphi}$ in \eqref{iso4} we obtain
$$ \langle \mathbb{V}_{\vec{\psi}} \vec{\phi}, \mathbb{V}_{\vec{\psi}} \vec{\varphi} \rangle_{{L^{2}(\mathbb{R}^2, \mathbb{H})}}=2 \langle \vec{\phi}, \vec{\varphi} \rangle_{L^{2}(\mathbb{R}, \mathbb{H}^{n+1})}.$$
\end{proof}	
	
\subsection{Reconstruction formula}
It is possible to recover the value of the signal if we know its true-poly QSTFT. In order to show a reconstruction formula we recall the following formula for the true quaternionic polyanalytic Bargmann (see Section 5). Let $ \varphi \in L^2(\mathbb{R},\mathbb{H})$, thus we have
\begin{equation}
\label{Bar1}
(B^{n+1} \varphi)(q)= 2^{\frac{3}{4}}(2^n n! (2 \pi)^n)^{- \frac{1}{2}} \int_{\mathbb{\mathbb{R}}} e^{- \pi(q^2+t^2)+2 \pi \sqrt{2} q t} H_n \left( \frac{q+ \bar{q}}{ \sqrt{2}}-t \right) \varphi(t) dt,
\end{equation}
where $H_n$ are the weighted Hermite polynomials, see \eqref{herm2}. 
\\We have the following relation between the weighted Hermite functions and the weighted Hermite polynomials
\begin{equation}
\label{Bar2}
h_n^{2 \pi}(u)=H_n(u) e^{- \pi u^2}.
\end{equation}
Before to prove the reconstruction formula we need the following auxiliary result.
\begin{lem}
\label{Bar5}
Let $ \varphi$ be a function in $L^2(\mathbb{R}, \mathbb{H})$ and $ \psi_n(t)=\frac{h_n^{2 \pi}(t)}{\|h_n^{2 \pi} \|_{L^2(\mathbb{R}, \mathbb{H})}}$. Then we have
\begin{equation}
\label{Bar3}
\mathcal{V}_{\psi_n} \varphi(x, \omega)= \sqrt{2} \int_{\mathbb{R}} e^{-2 \pi I \omega t} \psi_n(x-t) \varphi(t) dt.
\end{equation}
\end{lem}
\begin{proof}
From \eqref{QST1} and \eqref{Bar1} we get
\begin{eqnarray*}
\mathcal{V}_{\psi_n} \varphi(x, \omega) \! \! \!&=&\! \! \! \!  e^{-I \pi x \omega} \left( B^{n+1} \varphi  \right) \left( \frac{\bar{q}}{\sqrt{2}} \right) e^{- \frac{|q|^2 \pi}{2}}\\
&=&\! \! \! \! 2^{\frac{3}{4}}(2^n n! (2 \pi)^n)^{- \frac{1}{2}} e^{-I \pi x \omega} \! \! \! \int_{\mathbb{\mathbb{R}}} e^{- \pi\left( \frac{\bar{q}^2}{2}+t^2\right)+2 \pi \bar{q} t}   \\
&& \! \! \! \! \cdot H_n \left( \frac{\frac{\bar{q}}{\sqrt{2}}+ \frac{q}{\sqrt{2}}}{ \sqrt{2}}-t \right) \varphi(t) \, dt \, \, e^{- \frac{|q|^2 \pi}{2}}.\\
\end{eqnarray*}
Now, since $q=x+ I \omega$ we get
\begin{eqnarray*}
\mathcal{V}_{\psi_n} \varphi(x, \omega) \! \! \! &=& \! \! \!2^{\frac{3}{4}}(2^n n! (2 \pi)^n)^{- \frac{1}{2}} e^{-I \pi x \omega} e^{- \frac{ \pi x^2}{2}} e^{I \pi x \omega} e^{\frac{\omega^2 \pi}{2}}e^{- \frac{ \pi x^2}{2}} e^{-\frac{\omega^2 \pi}{2}} \cdot \\
&& \! \! \! \cdot \int_{\mathbb{\mathbb{R}}} e^{- \pi t^2+2 \pi x t -2 \pi I \omega t} H_n \left(x-t \right) \varphi(t)dt \\
&=& \! \! \! 2^{\frac{3}{4}}(2^n n! (2 \pi)^n)^{- \frac{1}{2}} \int_{\mathbb{R}} e^{-2 \pi I \omega t} e^{- \pi t^2+2 \pi xt -\pi x^2 } H_n(x-t) \varphi(t) dt\\
&=& \! \! \! 2^{\frac{3}{4}}(2^n n! (2 \pi)^n)^{- \frac{1}{2}} \frac{\|h_n^{2 \pi} \|_{L^2(\mathbb{R}, \mathbb{H})}}{\|h_n^{2 \pi} \|_{L^2(\mathbb{R}, \mathbb{H})}} \int_{\mathbb{R}}e^{-2 \pi I \omega t} e^{- \pi(x-t)^2} H_n(x-t) \varphi(t) dt.\\
\end{eqnarray*}
Hence by \eqref{Bar2} we get
\begin{eqnarray*}
\mathcal{V}_{\psi_n} \varphi(x, \omega)&=& 2^{\frac{3}{4}}(2^n n! (2 \pi)^n)^{- \frac{1}{2}} (2^n (2 \pi)^n n! 2^{- \frac{1}{2}})^{\frac{1}{2}} \int_{\mathbb{R}}e^{-2 \pi I \omega t} \frac{h_n^{2 \pi}(x-t)}{ \| h_n^{2 \pi} \|_{L^2( \mathbb{R}, \mathbb{H})}} \varphi(t) dt\\
&=& \sqrt{2} \int_{\mathbb{R}} e^{- 2 \pi I \omega t} \psi_n(x-t) \varphi(t) dt.
\end{eqnarray*}
This concludes the proof.
\end{proof}
It is possible to have something similar also for the full-poly QSFT.
\begin{cor}
Let $ \vec{\varphi}=( \varphi_0,..., \varphi_n) $ be a vector-valued function in $L^{2}( \mathbb{R}, \mathbb{H}^{n+1}).$ Then for $ \vec{\psi}=(\psi_0,...,\psi_n)$ we have
$$ \mathbb{V}_{\vec{\psi}} \vec{\varphi}(x, \omega)= \sqrt{2} \int_{\mathbb{R}} e^{-2 \pi I \omega t} \sum_{j=0}^n \psi_j(x-t) \varphi_j(t) dt.$$
\end{cor}
\begin{proof}
By Proposition \ref{sum2} we know that
$$ \mathbb{V}_{\vec{\psi}} \vec{\varphi}(x, \omega)= \sum_{j=0}^n \mathcal{V}_{\psi_j} \varphi_j(x, \omega).$$
For each member of the sum we know that the equality \eqref{Bar3} holds. So we have
$$ \mathbb{V}_{\vec{\psi}} \vec{\varphi}(x. \omega)= \sqrt{2} \int_{\mathbb{R}} e^{-2 \pi I \omega t} \sum_{j=0}^n \psi_j(x-t) \varphi_j(t) dt.$$
\end{proof}
Now, we are ready to prove the reconstruction formula for the true-poly QSTFT.
\begin{thm}
Let $ \varphi \in L^2( \mathbb{R}, \mathbb{H})$. Then for all $ y \in \mathbb{R}$  we have 
\begin{equation}
\label{reco}
\varphi(y)= \frac{1}{\sqrt{2}} \int_{\mathbb{R}^2} e^{2 \pi I \omega y} [\mathcal{V}_{\psi_n} \varphi(x, \omega)] \psi_n(x-y) d x d \omega.
\end{equation}
\end{thm}
\begin{proof}
Let us set
$$ \phi(y):=\frac{1}{\sqrt{2}} \int_{\mathbb{R}^2} e^{2 \pi I \omega y} [\mathcal{V}_{\psi_n} \varphi(x, \omega)] \psi_n(x-y) d x d \omega, \qquad \forall y \in \mathbb{R}.$$
Let $ \Theta \in L^2( \mathbb{R}, \mathbb{H})$. By Fubini's theorem, Lemma \ref{Bar5} and the Moyal formula \eqref{iso1} we get
\begin{eqnarray*}
\langle \phi, \Theta \rangle_{L^2( \mathbb{R}, \mathbb{H})}&=& \int_{\mathbb{R}} \overline{\Theta(y)} \phi(y) dy\\
&=& \frac{1}{\sqrt{2}} \int_{\mathbb{R}^3} \overline{\Theta(y)}e^{2 \pi I \omega y} [\mathcal{V}_{\psi_n} \varphi(x, \omega)] \psi_n(x-y)   d x d \omega dy\\
&=& \frac{1}{2} \int_{\mathbb{R}^2} \left( \sqrt{2} \int_{\mathbb{R}}\overline{e^{-2 \pi I \omega y} \psi_n(x-y) \Theta(y) dy} \right) \mathcal{V}_{\psi_n} \varphi(x, \omega) dx d \omega\\
&=& \frac{1}{2} \int_{\mathbb{R}^2} \overline{\mathcal{V}_{\psi_n} \Theta(x, \omega)}\mathcal{V}_{\psi_n} \varphi(x, \omega) dx d \omega\\
&=& \frac{1}{2}  \langle \mathcal{V}_{\psi_n} \varphi, \mathcal{V}_{\psi_n} \Theta \rangle_{L^2( \mathbb{R}^2, \mathbb{H})}=\langle \varphi, \Theta \rangle_{L^2( \mathbb{R}, \mathbb{H})}.
\end{eqnarray*}
Therefore, since for all $ \Theta \in L^2( \mathbb{R}, \mathbb{H})$ we have $\langle \phi, \Theta \rangle_{L^2( \mathbb{R}, \mathbb{H})}=\langle \varphi, \Theta \rangle_{L^2( \mathbb{R}, \mathbb{H})}$ we conclude that
$$ \varphi(y)=\phi(y)=\frac{1}{\sqrt{2}} \int_{\mathbb{R}^2} e^{2 \pi I \omega y} [\mathcal{V}_{\psi_n} \varphi(x, \omega)] \psi_n(x-y) d x d \omega .$$
\end{proof}
\begin{rem}
It is possible to have a kind of reconstruction formula also for the full-poly QSTFT. Basically, we use the reconstruction formula \eqref{reco} for each component of the vector-valued function $ \vec{\varphi}=(\varphi_0,...,\varphi_n)$. Thus for $ 0 \leq j \leq n$ we have
$$ \varphi_j(y)=\frac{1}{\sqrt{2}} \int_{\mathbb{R}^2} e^{2 \pi I \omega y} [\mathcal{V}_{\psi_j} \varphi_j(x, \omega)] \psi_j(x-y) d x d \omega \qquad \forall y \in \mathbb{R}.$$
\end{rem}
\begin{rem}
If $n=0$ in \eqref{reco} we obtain the formula proved in \cite[Thm. 5.8]{DMD}
\end{rem}
Both the true-poly QTSFT and the full poly one admit a left-side inverse, which is possible to compute.
\begin{thm}
\label{adj5}
Let $\Lambda \in L^2(\mathbb{R}^2, \mathbb{H})$. For all $y \in \mathbb{R}$ we define $ \mathcal{V}^{*}_{\psi_n}$ the adjoint operator of $ \mathcal{V}_{\psi_n}$ as
\begin{equation}
\label{adj}
\mathcal{V}^{*}_{\psi_n} (\Lambda) (y)= \sqrt{2} \int_{\mathbb{R}^2} e^{2 \pi I \omega y} \Lambda(x, \omega) \psi_n(x-y) dx d \omega.
\end{equation}
Moreover,
\begin{equation}
\label{adj1}
\mathcal{V}^{*}_{\psi_n}\mathcal{V}_{\psi_n}=2 Id.
\end{equation}
\end{thm}
\begin{proof}
Firstly we show that $ \mathcal{V}^{*}_{\psi_n}$ is the adjoint operator of $ \mathcal{V}_{\psi_n}$. Let $ h \in L^2( \mathbb{R}, \mathbb{H})$. The application of Fubini's theorem and formula \eqref{Bar3} lead to 
\begin{eqnarray*}
\langle \mathcal{V}^*_{\psi_n}(\Lambda),h \rangle_{L^2(\mathbb{R}, \mathbb{H})}&=& \int_{\mathbb{R}} \overline{h(y)} \mathcal{V}^*_{\psi_n}(\Lambda)(y) dy\\
&=& \sqrt{2} \int_{\mathbb{R}^3}\overline{h(y)}e^{2 \pi I \omega y} \Lambda(x, \omega) \psi_n(x-y) dx d \omega dy\\
&=& \int_{\mathbb{R}^2} \sqrt{2}  \left( \int_{\mathbb{R}} \overline{e^{-2 \pi I \omega y} \psi_n(x-y) h(y)}dy \right) \Lambda(x, \omega) dx d \omega\\
&=& \int_{\mathbb{R}^2} \overline{\mathcal{V}_{\psi_n} h(x, \omega)} \Lambda(x, \omega) dx d \omega= \langle \Lambda, \mathcal{V}_{\psi_n} h \rangle_{L^2({\mathbb{R}^2}, \mathbb{H})}.
\end{eqnarray*}
Now, we prove \eqref{adj1}. It follows by formula \eqref{reco}
\begin{eqnarray*}
\mathcal{V}_{\psi_n}^*(\mathcal{V}_{\psi_n} \varphi)(y)&=& \sqrt{2} \int_{\mathbb{R}^2} e^{2 \pi I \omega y} [\mathcal{V}_{\psi_n} \varphi(x, \omega)] \psi_n(x-y) dx d \omega\\
&=& 2 \varphi(y).
\end{eqnarray*}
\end{proof}
\begin{rem}
The formula \eqref{adj} is the left-side inverse of the true-poly QSTFT.
\end{rem}
We can prove a similar result for the full-poly QSTFT.
\begin{thm}
For  $ \Lambda \in L^2( \mathbb{R}^2, \mathbb{H})$ and for all $y \in \mathbb{R}$ we define the full-poly adjoint operator $\mathbb{V}_{\vec{\psi}}^* : L^2( \mathbb{R}^2, \mathbb{H}) \to L^2(\mathbb{R}, \mathbb{H}^{n+1})$ as
\begin{equation}
\label{vec1}
\mathbb{V}^*_{\vec{\psi}} \Lambda= \left( \mathcal{V}^*_{\psi_0} \Lambda, ..., \mathcal{V}^*_{\psi_n} \Lambda  \right),
\end{equation}
where $\vec{\psi}=(\psi_0,..., \psi_n)$ and $ \mathcal{V}^*_{\psi_j} \Lambda$ are defined as follow
$$ \mathcal{V}^{*}_{\psi_j} \Lambda= \sqrt{2} \int_{\mathbb{R}^2} e^{2 \pi I \omega y} \Lambda(x, \omega) \psi_j(x-y) dx d \omega \qquad 0 \leq j \leq n.$$
 Moreover,
\begin{equation}
\label{adj4}
\mathbb{V}_{\vec{\psi}}^* \mathbb{V}_{\vec{\psi}}=2 Id_{L^2(\mathbb{R}, \mathbb{H}^{n+1})}.
\end{equation}
\end{thm}
\begin{proof}
The construction of the vector \eqref{vec1} follows from the definition of the full-poly adjoint operator and Theorem \ref{adj5}.
Finally, \eqref{adj4} follows from the fact that the full-poly QSTFT is an isometric operator, see Theorem \ref{iso5}.
\end{proof}	
	
\begin{rem}
The vectorial operator proposed in formula \eqref{vec1} can be considered a left-side inverse of the full-poly QSTFT.
\end{rem}
\subsection{Reproducing kernel property}
Now, we will write the reproducing kernel of the quaternionic Gabor space associated to the true-poly QSTFT. We define it as   
$$\mathcal{G}_{\mathbb{H}}^{\psi_n}:=\lbrace{\mathcal{V}_{\psi_n}\varphi, \text{  } \varphi\in L^2(\mathbb{R},\mathbb{H})}\rbrace.$$
We also consider a vector-valued version of the previous space which is given by 

$$\mathbb{G}_{\mathbb{H}}^{\vec{\psi}}=\lbrace{\mathbb{V}_{\vec{\psi}}\vec{\varphi}, \textbf{ } \vec{\varphi}\in L^2(\mathbb{R}, \mathbb{H}^{n+1})}\rbrace,$$
where $ \vec{\varphi}=(\varphi_0,..., \varphi_n)$ and $ \vec{\psi}=(\psi_0,...,\psi_n)$.
\begin{thm}
\label{cor3}
Let $\varphi$ be in $L^2(\mathbb{R},\mathbb{H})$ and $\psi_n(t)=\frac{h_n^{2 \pi}(t)}{\|h_n^{2 \pi}\|_{L^2(\mathbb{R}^2,\mathbb{H})}}$. If 
$$\displaystyle K_{\psi_n}(x,\omega;x',\omega')=\int_{\mathbb{R}}e^{2\pi I(w'-w)t}\psi_n(x'-t)\psi_n(x-t)dt.$$
Then, $K_{\psi_n}$ is the reproducing kernel of the true-poly QSTFT, i.e
$$\mathcal{V}_{\psi_n}\varphi(x',w')=\scal{\mathcal{V}_{\psi_n}\varphi, K_{\psi_n}}_{L^2(\mathbb{R}^2,\mathbb{H})}.$$
\end{thm}
\begin{proof}
We use Lemma \ref{Bar5}, the inversion formula \eqref{reco} and Fubini's theorem to get 
\begin{eqnarray*}
\mathcal{V}_{\psi_n}\varphi(x',w') \! \! \! \!&=&\! \! \! \!\sqrt{2}\int_{\mathbb{R}} e^{-2\pi Iw't} \varphi(t)\psi_n(x'-t)dt\\
& =& \! \! \! \!\int_{\mathbb{R}} e^{-2\pi Iw' t}\psi_n(x'-t)\left(\int_{\mathbb{R}^2}e^{2\pi I wt}[\mathcal{V}_{\psi_n}\varphi(x,w)]  \psi_n(x-t) dx dw\right)dt\\
& =& \! \! \! \! \int_{\mathbb{R}^3} e^{-2\pi I w't}\psi_n(x'-t)e^{2\pi I wt}
\psi_n(x-t)\mathcal{V}_{\psi_n}\varphi(x,w)dxdwdt\\
& =&\! \! \! \! \int_{\mathbb{R}^2}\left(\int_\mathbb{R} e^{-2\pi It(w'-w)}\psi_n(x'-t)\psi_n(x-t)dt\right)\mathcal{V}_{\psi_n}\varphi(x,w)dxdw\\
& = & \! \! \! \! \int_{\mathbb{R}^2}\overline{K_{\psi_n}(x,\omega;x',\omega')}\mathcal{V}_{\psi_n}\varphi(x,w)dx dw\\
& =& \! \! \! \! \scal{\mathcal{V}_{\psi_n}\varphi, K_{\psi_n}}_{L^2(\mathbb{R}^2,\mathbb{H})}.
\end{eqnarray*}	
\end{proof}
A similar theorem holds for the full-poly QSTFT, that we can state as follows.
\begin{thm}
\label{cor4}
Let $\vec{\varphi}=(\varphi_0,..., \varphi_n )\in L^2(\mathbb{R},\mathbb{H}^{n+1})$. The following functions $$\displaystyle K_{\vec{\psi}}(x, \omega; x', \omega')= \sum_{j=0}^n \int_{\mathbb{R}}e^{2\pi I(w'-w)t}\psi_j(x'-t)\psi_j(x-t)dt$$
are the reproducing kernel of the space $\mathbb{G}_{\mathbb{H}}^{\vec{\psi}}$, i.e:
$$\mathbb{V}_{\vec{\psi}}\vec{\varphi}(x',w')=\scal{\mathbb{V}_{\psi_N}\vec{\varphi}, K_{\vec{\psi}} }_{L^2(\mathbb{R}^2,\mathbb{H}^{n+1})}.$$
	
\end{thm}
\begin{proof}
It follows from the previous theorem and the fact that we can write the full-poly QSTFT as sums of true-poly QSTFTs (Proposition \ref{sum2}) and the definition of inner product of the space $L^2(\mathbb{R}^2, \mathbb{H}^{n+1})$ (see \eqref{inner}). In particular, we have 
\[ \begin{split}
\displaystyle \mathbb{V}_{\vec{\psi}}\vec{\varphi}(x',w')&=\sum_{j=0}^n \mathcal{V}_{\psi_j}\varphi_j(x',w') \\
&=\sum_{j=0}^{n}\scal{\mathcal{V}_{\psi_j}\varphi_j, K_{\psi_j}}_{L^2(\mathbb{R}^2,\mathbb{H})}\\
& = \scal{\mathbb{V}_{\vec{\psi}}\vec{\varphi}, K_{\vec{\psi}} }_{L^2(\mathbb{R}^2,\mathbb{H}^{n+1})}.
\\ 
\end{split}
\]

\end{proof}
\begin{rem}
If $n=0$ in Theorem \ref{cor3} and Theorem \ref{cor4} we recover the same result of \cite[Thm.5.12]{DMD}.
\end{rem}

\subsection{Lieb's uncertainty principle}
In this section we want to extend to the quaternionic polyanalytic theory the Lieb's uncertainty principle. Let us recall that the uncertainty principles state that a signal cannot be simultaneously sharply located both in time and frequency domains.  This is emphasised by the following generic principle \cite{G}:
\newline
\begin{center}
\emph{"A function cannot be concentrated on small sets in the time-frequency plane, no matter which  time-frequency representation is used."}
\end{center}
\begin{thm}{(Weak uncertainty principle)}
\label{Wun1}
Let $ \varphi \in L^{2}( \mathbb{R},\mathbb{H})$ be an unit vector, $U$ an open set of $ \mathbb{R}^2$ and $ \varepsilon \geq 0$ such that
\begin{equation}
\label{ipo1}
\int_{U} | \mathcal{V}_{\psi_n} \varphi(x, \omega)|^2 dx d \omega \geq 1- \varepsilon,
\end{equation}
then
$|U| \geq \frac{1- \varepsilon}{2}.$
\end{thm}
\begin{proof}
From the definition of true-poly QSTFT and Proposition \ref{Kamal} we have 
\begin{eqnarray}
\label{est1}
| \mathcal{V}_{\psi_n} \varphi(x, \omega)| \! \! \! \! &=&  \! \! \! \! | e^{-I \pi x \omega}| \left| B^{n+1}\left( \frac{\bar{q}}{\sqrt{2}}\right)\right| e^{-\frac{|q|^2 \pi}{2}}\\ \nonumber
& \leq& \! \! \! \! \sqrt{2} e^{\frac{|q|^2 \pi}{2}} \| \varphi \|_{L^2(\mathbb{R},\mathbb{H})} e^{-\frac{|q|^2 \pi}{2}}\\ \nonumber
&=& \! \! \! \! \sqrt{2}. \nonumber
\end{eqnarray}
Thus by \eqref{ipo1} we get
$$ 1- \varepsilon \leq \int_{U} | \mathcal{V}_{\psi_n} \varphi(x, \omega)|^2 dx d \omega  \leq \| \mathcal{V}_{\psi_n} \varphi \|_{\infty}^2 |U| \leq 2 |U|.$$
Hence
$$|U| \geq \frac{1- \varepsilon}{2}.$$
\end{proof}
A weak uncertainty principle holds also for the full-poly QSTFT.
\begin{thm}
\label{Wun2}
Let $ \vec{\varphi}=(\varphi_0,...,\varphi_n)$ be a vector valued function in $L^{2}(\mathbb{R}, \mathbb{H}^{n+1})$ with $ \| \varphi_j \|_{L^2(\mathbb{R}, \mathbb{H})}=1$, for all $1 \leq j \leq n$, and $ \vec{\psi}=(\psi_0,...,\psi_n)$ be a vector-valued window function. If $U$ is an open set of $ \mathbb{R}^2$ and $ \varepsilon \geq 0$ we suppose
\begin{equation}
\label{ipo2}
\int_{U} | \mathbb{V}_{\vec{\psi}} \vec{\varphi}(x, \omega)|^2 dx d \omega \geq 1- \varepsilon,
\end{equation}
then $ |U| \geq \frac{1- \varepsilon}{2(n+1)^2 }.$
\end{thm}
\begin{proof}
By Proposition \ref{sum2} we know that
$$\mathbb{V}_{\vec{\psi}}\vec{\varphi}= \sum_{j=0}^n \mathcal{V}_{\psi_j} \varphi_j(x, \omega).$$
Thus by the estimate \eqref{est1} applied at each members of the sum turns out that
\begin{equation}
\label{est2}
\left| \mathbb{V}_{\vec{\psi}}\vec{\varphi}(x, \omega)\right|\leq \sum_{j=0}^n \left|\mathcal{V}_{\psi_j} \varphi_j(x, \omega)\right| \leq \sqrt{2} (n+1).
\end{equation}
Therefore by \eqref{ipo2}  and \eqref{est2} we obtain
$$ 1- \varepsilon \leq \int_{U} | \mathbb{V}_{\vec{\psi}} \vec{\varphi}(x, \omega)|^2 dx d \omega  \leq \| \mathbb{V}_{\vec{\psi}} \vec{\varphi} \|_{\infty}^2|U| \leq 2|U|(n+1)^2.$$
Then
$$ |U| \geq \frac{1- \varepsilon}{2(n+1)^2 }.$$
\end{proof}
In order to improve the estimates of the weak uncertainty principles we need the following $L^p$- estimate of the true-poly QSTFT and the full-poly one. For the first one we omit the proof since it can be shown with exactly the same computations of \cite[Thm. 5.14]{DMD}.
\begin{prop}
Let $ \varphi \in L^2(\mathbb{R}, \mathbb{H})$ and $p \in [2, \infty)$ then we have
\begin{equation}
\label{Lieb}
\int_{\mathbb{R}^2}| \mathcal{V}_{\psi_n} \varphi(x, \omega)|^p dx d \omega \leq \frac{2^{p+1}}{p} \| \varphi \|_{L^2(\mathbb{R}, \mathbb{H})}^p.
\end{equation}
\end{prop}
\begin{prop}
\label{Lieb2}
Let $ \vec{\varphi}=(\varphi_0,...,\varphi_n)$ be a vector valued function in $L^{2}( \mathbb{R}, \mathbb{H}^{n+1})$. For $p \in [2, \infty)$ we have
\begin{equation}
\label{Lieb1}
\int_{\mathbb{R}^2}| \mathbb{V}_{\vec{\psi}} \vec{\varphi}(x, \omega)|^p dx d \omega \leq \frac{2^{p+1}}{p} (n+1)^{p-1} \| \vec{\varphi}\|_{L^2(\mathbb{R}, \mathbb{H}^{n+1})}^p.
\end{equation}
\end{prop}
\begin{proof}
From \eqref{sum1} we obtain
$$ | \mathbb{V}_{\vec{\psi}} \vec{\varphi}(x, \omega)|^p \leq \left( \sum_{j=0}^n | \mathcal{V}_{\psi_j} \varphi_j(x, \omega)| \right)^p \leq (n+1)^{p-1} \sum_{j=0}^n | \mathcal{V}_{\psi_j} \varphi_j(x. \omega)|^p. $$
Now, we integrate and apply \eqref{Lieb} on each members of the sum and we get
\begin{eqnarray*}
\int_{\mathbb{R}^2}| \mathbb{V}_{\vec{\psi}} \vec{\varphi}(x, \omega)|^p dx d \omega \! \! \! \! &\leq& \! \! \! \! (n+1)^{p-1} \sum_{j=0}^n  \int_{\mathbb{R}^2}| \mathcal{V}_{\psi_j} \varphi_j(x, \omega)|^p dx d \omega\\
& \leq & \! \! \! \! \frac{2^{p+1}}{p} (n+1)^{p-1} \sum_{j=0}^n \| \varphi_j \|^p_{L^{2}(\mathbb{R}, \mathbb{H})}\\
& \leq& \! \! \! \! \frac{2^{p+1}}{p} (n+1)^{p-1} \left( \sum_{j=0}^n \| \varphi_j \|^{2}_{L^{2}(\mathbb{R}, \mathbb{H})} \right)^{\frac{p}{2}}\\
&=& \! \! \! \! \frac{2^{p+1}}{p} (n+1)^{p-1} \| \vec{\varphi}\|_{L^2(\mathbb{R}, \mathbb{H}^{n+1})}^p.
\end{eqnarray*}
\end{proof}
Next, we show that the inequalities \eqref{Lieb} and \eqref{Lieb1} yield a sharper estimates for the Theorem \ref{Wun1} and Theorem \ref{Wun2}, respectively.
\begin{thm}
\label{Wun3}
Let us consider $ \vec{\varphi}=( \varphi_0,..., \varphi_n) \in L^2( \mathbb{R}, \mathbb{H}^{n+1})$, such that $ \| \varphi_j \|_{L^2(\mathbb{R}, \mathbb{H})}=1$, for any $ 0 \leq j \leq n$. Let $U$ be an open set of $ \mathbb{R}^2$, $ \varepsilon \geq 0$ and
\begin{equation}
\int_{U} | \mathbb{V}_{\vec{\psi}} \vec{\varphi}(x, \omega)|^2 dx d \omega \geq 1- \varepsilon,
\end{equation}
then $$ |U| \geq \left( \frac{2^{p+1}}{p} \right)^{- \frac{2}{p-2}} (1- \varepsilon)^{ \frac{p}{p-2}} (n+1)^{\frac{2-3p}{p-2}}, \qquad \hbox{for} \quad p>2.$$
\end{thm}
\begin{proof}
The Hölder's inequality and the estimate \eqref{Lieb1} imply that
\begin{eqnarray*}
1- \varepsilon \! \! \! \! & \leq & \! \! \! \! \int_{U} | \mathbb{V}_{\vec{\psi}} \vec{\varphi}(x, \omega)|^2 dx d \omega = \int_{\mathbb{R}^2} | \mathbb{V}_{\vec{\psi}} \vec{\varphi}(x, \omega)|^2 \chi_U(x, \omega) dx d \omega \\
& \leq & \! \! \! \! \left( \int_{\mathbb{R}^2} | \mathbb{V}_{\vec{\psi}} \vec{\varphi}(x, \omega)|^p dx d \omega \right)^{\frac{2}{p}} |U|^{\frac{p-2}{p}}\\
& \leq & \! \! \! \! \left( \frac{2^{p+1}}{p} \right)^{\frac{2}{p}} (n+1)^{\frac{3p-2}{p}}|U|^{\frac{p-2}{p}}.
\end{eqnarray*}
Therefore
$$ |U| \geq \left( \frac{2^{p+1}}{p} \right)^{- \frac{2}{p-2}} (1- \varepsilon)^{ \frac{p}{p-2}} (n+1)^{\frac{2-3p}{p-2}}, \qquad \hbox{for} \quad p>2.$$
\end{proof}
\begin{thm}
\label{Wun4}
Let $ \varphi$ be in $L^2( \mathbb{R}, \mathbb{H})$. If we assume that $U$ is an open set of $ \mathbb{R}^2$, $ \varepsilon \geq 0$ and $\| \varphi \|_{L^2(\mathbb{R}, \mathbb{H})}=1$ such that
$$ \int_U | \mathcal{V}_{\psi_n} \varphi (x, \omega)|^2 dx d \omega \geq 1- \varepsilon.$$
Then, we have
$$ |U| \geq \left( \frac{2^{p+1}}{p} \right)^{- \frac{2}{p-2}} (1- \varepsilon)^{ \frac{p}{p-2}}, \qquad \hbox{for} \quad p>2.$$
\end{thm}
\begin{proof}
It follows by using similar techniques of Theorem \ref{Wun3}. 
\end{proof}
\begin{rem}
If $n=0$ in Theorem \ref{Wun3} and Theorem \ref{Wun4}, these results coincide with \cite[Thm. 5.15]{DMD}.
\end{rem}

\section{Appendices}
We add the following appendices because of the lack of references. We prove an orthogonality relation for the complex Hermite polynomials, with a general parameter $ \alpha>0$. Then, we show some basic properties of the Hermite polynomials, for a general parameter $ \nu>0$.
\subsection{Appendix A}
We consider the complex Hermite polynomials defined by 

$$H^{\alpha}_{m,p}(z,\overline{z})=(-1)^{m+p}e^{\alpha |z|^2}\dfrac{\partial^{m+p}}{\partial z^m \partial \overline{z}^p}\left( e^{-\alpha|z|^2}\right), \qquad \alpha >0.$$
First, using some direct calculations we observe that we have 

$$H_{0,p}^{\alpha}(z,\overline{z})=\alpha^p z^p$$
and $$H_{1,p}^{\alpha}(z,\overline{z})=\alpha^{p+1}\overline{z}z^p-\alpha^p pz^{p-1}.$$ 
In order to revise the calculations of the complex Hermite polynomials norm in \\$L^{2,\alpha}(\mathbb{C}):=L^2(\mathbb{C},e^{-\alpha |z|^2}dA(z))$ we follow the ideas of \cite{II}.
For this, let us consider the operator given by 
$$A:=-\dfrac{\partial}{\partial z}+\alpha \overline{z}.$$
Then, we can prove the following result.
\begin{lem}
For all $m\in\mathbb{N}$, we have 
$$A^m\left((\alpha z)^p\right)=H_{m,p}^{\alpha}(z,\overline{z}).$$
\end{lem}
\begin{proof}
We use an induction process to prove this result. 
Firstly, for $m=1$ we have 
\[ \begin{split}
\displaystyle A\left((\alpha z)^p\right) & =-\alpha^p pz^{p-1}+\alpha^{p+1}\overline{z}z^p \\
&=H_{1,p}^{\alpha}(z,\overline{z}).
\end{split}
\]
Now, let us suppose that this relation holds for $m$ and  prove it for $m+1$. Indeed, we use the induction hypothesis combined with the Leibniz rule to get 
\[ \begin{split}
\displaystyle A^{m+1}\left((\alpha z)^p\right) & = A\left(H_{m,p}^{\alpha}(z,\overline{z}) \right)\\
&= -(-1)^{m+p}\left(\alpha\overline{z}e^{\alpha|z|^2}\dfrac{\partial^{m+p}}{\partial z^m\partial \overline{z}^p} e^{-\alpha|z|^2}+e^{\alpha|z|^2}\dfrac{\partial^{m+p+1}}{\partial z^{m+1}\partial \overline{z}^p}e^{-\alpha |z|^2}\right)\\
& +(-1)^{m+p}\alpha\overline{z}e^{\alpha|z|^2}\dfrac{\partial^{m+p}}{\partial z^m\partial \overline{z}^p} e^{-\alpha|z|^2}
\\
& =H_{m+1,p}^{\alpha}(z,\overline{z}).
\end{split}
\]
Thus, the result holds by induction, this ends the proof.
\end{proof}
\begin{thm}
\label{new2}
Let $\alpha>0$ and $m,p\in\mathbb{N}$. Then, we have 
$$||H_{m,p}^{\alpha}(z,\overline{z})||_{L^{2,\alpha}(\mathbb{C})}=\alpha^{p+m-1}\pi m!p! \,.$$
\end{thm}
\begin{proof}
We set $\varphi_p(z)=(\alpha z)^p$, then using direct computations we obtain  
\[ \begin{split}
\displaystyle H_{m,p}^{\alpha}(z,\overline{z}) & = A^m\left(\varphi_p(z) \right)\\
&=\left(-\dfrac{\partial}{\partial z}+\alpha \overline{z}\right)^{m}(\varphi_p(z)) \\
&=\alpha^p\sum_{j=0}^m (-1)^j{m \choose j}\left(\dfrac{\partial^j}{\partial z^j}M_{\overline{z}}^{m-j}(z^p)\right)\alpha^{m-j},
\end{split}
\]
where the conjugate multiplication operator is given by $M_{\overline{z}}f=\overline{z}f$. Then, using the fact that $$\displaystyle \left(\dfrac{\partial}{\partial z}\right)^j z^p=\frac{\Gamma(p+1)}{\Gamma(p-j+1)}z^{p-j},  \quad j=0,1,2,...$$
we obtain
\[ \begin{split}
\displaystyle H^{\alpha}_{m,p}(z,\overline{z}) & = \alpha^p m!\sum_{j=0}^m(-1)^j\frac{p!}{j!(m-j)!(p-j)!}\alpha^{m-j}z^{p-j}\overline{z}^{m-j}.
\end{split}
\]
Then, we pass to the polar coordinates $z=re^{i\theta}$ with $r\geq 0$ and $\theta\in[0,2\pi]$ and get $$\displaystyle H^{\alpha}_{m,p}(re^{i\theta},re^{-i\theta})=\alpha^p m!e^{i\theta(p-m)}\sum_{j=0}^{m}\dfrac{p!(-1)^j}{j!(m-j)!(p-j)!}\alpha^{m-j}r^{p+m-2j}.$$
We change the summation index to $k=m-j$, so we get
\begin{eqnarray}
\label{new4}
H_{m,p}^{\alpha}(re^{i\theta}, re^{-i\theta})&=& \alpha^p m!e^{i\theta(p-m)}\\  
&& \cdot \sum_{k=0}^{m}\dfrac{p!(-1)^{m-k}}{(m-k)!k!(p-m+k)!}\alpha^{k}r^{p-m+2k}.\nonumber
\end{eqnarray}

After that we use the classical formula for the generalized Laguerre polynomials given by 
$$L^{\beta}_{m}(x):=\displaystyle \sum_{k=0}^m(-1)^k{m+\beta \choose m-k} \frac{x^k}{k!}.$$
Thus, if $\beta:=p-m$, with $p>m$ we have 
$$L_{m}^{p-m}(x):=\displaystyle \sum_{k=0}^m(-1)^k \frac{p!}{(m-k)!(p-m+k)!}\frac{x^k}{k!}.$$
In particular, from the formula \eqref{new4} we get $$\displaystyle H_{m,p}^{\alpha}(re^{i\theta},re^{-i\theta})=\alpha^p m!(-1)^me^{i\theta(p-m)}r^{p-m}L^{p-m}_{m}(\alpha r^2).$$
Now, we compute the orthogonality relation using the Fubini's theorem. Let $m',p' \in\mathbb{N}$ we have 
\begin{eqnarray*}
&& \scal{H_{m,p}^{\alpha}(z, \bar{z}),H_{m',p'}^{\alpha}(z, \bar{z})}_{L^{2, \alpha}(\mathbb{C})}  =\int_{\mathbb{C}}H_{m,p}(z,\overline{z})\overline{H_{m',p'}(z,\overline{z})}e^{-\alpha|z|^2}dA(z)\\ 
&& =\alpha^{2p}(m!)(m'!) (-1)^m (-1)^{m'}\left(\int_{0}^{2\pi}e^{i\theta(p-m)}e^{-i\theta(p'-m')}d\theta\right)  \\
&& \cdot \left(\int_{0}^{\infty}r^{p-m} r^{p'-m'} r L^{p-m}_{m}(\alpha r^2) L^{p'-m'}_{m'}(\alpha r^2)e^{-\alpha r^2}dr \right).\\ 
\end{eqnarray*}

We set $\ell:=p-m$ and $\ell'=p'-m'$. Thus we get
\begin{eqnarray*}
&& \scal{H_{m,p}^{\alpha}(z, \bar{z}),H_{m',p'}^{\alpha}(z, \bar{z})}_{L^{2, \alpha}(\mathbb{C})}  = 2\pi \alpha^{2p}(m!)(m'!)  \cdot \\
&& \cdot (-1)^m (-1)^{m'} \delta_{\ell,\ell'}   \left(\int_{0}^{\infty}r^{p-m+1} r^{p'-m'}L^{p-m}_{m}(\alpha r^2) L^{p'-m'}_{m'}(\alpha r^2)e^{-\alpha r^2}dr \right).
\end{eqnarray*}
Since $\ell=p-m=p'-m'=\ell'$ we derive that $p=m+\ell$ and $p'=m'+\ell$. Therefore
\begin{eqnarray*}
&&\scal{H_{m,p}^{\alpha}(z, \bar{z}),H_{m',p'}^{\alpha}(z, \bar{z})}_{L^{2, \alpha}(\mathbb{C})}=
\scal{H_{m,m+\ell}^{\alpha}(z, \bar{z}),H_{m',m'+\ell}^{\alpha}(z, \bar{z})}_{L^{2, \alpha}(\mathbb{C})}\\
&& = 2\pi \alpha^{2p}(m!)(m'!)(-1)^m (-1)^{m'}  \left(\int_{0}^{\infty}r^{\ell+1} r^{\ell}L^{\ell}_{m}(\alpha r^2) L^{\ell}_{m'}(\alpha r^2)e^{-\alpha r^2}dr \right)\\
&& = 2\pi \alpha^{2p}(m!)(m'!) (-1)^m (-1)^{m'} \left(\int_{0}^{\infty}r^{2 \ell+1} L^{\ell}_{m}(\alpha r^2) L^{\ell}_{m'}(\alpha r^2)e^{-\alpha r^2}dr \right).\\
\end{eqnarray*}

We know that (see \cite{GR} pag. 809 paragraph 7.414 formula n° 3) $$\displaystyle\int_0^\infty L^{\gamma}_{k}(t)L^{\gamma}_{j}(t)t^\gamma e^{-t}dt=\dfrac{\Gamma(\gamma+k+1)}{k!}\delta_{k,j}.$$
	
Then, we use the following change of variables $s=\alpha r^2$ and get 
	
\[ \begin{split}
\displaystyle \scal{H_{m,p}^\alpha(z, \bar{z}),H_{m',p'}^\alpha(z, \bar{z})}_{L^{2, \alpha}(\mathbb{C})}  =&\pi \alpha^{2p-1}(m!)(m'!) (-1)^m (-1)^{m'} \cdot\\ 
& \cdot \int_0^\infty \left(\dfrac{s}{\alpha}\right)^{\ell} L^{\ell}_{m}(s) L^{\ell}_{m'}(s)e^{-s}ds\\ 
 = & \pi \alpha^{2p-1- \ell}(m!)^2\dfrac{\Gamma(m+\ell+1)}{m!} \delta_{m,m'}.
\end{split}
\]	
Since $ \ell=p-m$ and $p-m=p'-m'$ we get
\begin{eqnarray*}
\scal{H_{m,p}^{\alpha}(z, \bar{z}),H_{m',p'}^{\alpha}(z, \bar{z})}_{L^{2, \alpha}(\mathbb{C})} \! \! \! &=& \! \! \! \pi \alpha^{p+m-1} m! \Gamma(p+1) \delta_{p,p'} \delta_{m,m'}\\
&=& \! \! \!\pi \alpha^{p+m-1} m! p! \delta_{m,m'} \delta_{p,p'}.
\end{eqnarray*}

Therefore
$$ \| H_{m,p}^{\alpha} (z, \bar{z})\|_{L^{2, \alpha}(\mathbb{C})}^2=\pi \alpha^{p+m-1} m! p! \quad.$$
\end{proof}
\begin{rem}
\label{ort1}
In particular, for $\alpha=2\pi$ we obtain
$$ \displaystyle \scal{H_{m,p}^{2 \pi}(z, \bar{z}),H_{m',p'}^{2 \pi}(z, \bar{z})}_{L^{2, \alpha}(\mathbb{C})} =\frac{m!p!(2\pi)^{p+m}}{2}\delta_{m,m'} \delta_{p,p'}.$$
\end{rem}
\subsection{Appendix B}
Let us consider the following function for $ \nu>0$
$$\displaystyle W(x,t)=e^{-\frac{\nu}{2}t^2+\nu\sqrt{2}xt}=\sum_{n=0}^{\infty}\frac{H_n^\nu(x)}{n!}\frac{t^n}{2^{\frac{n}{2}}},$$
where $H_n^\nu$ are the weighted Hermite polynomials defined by
$$ H_n^\nu(x)=(-1)^n e^{\nu y^2} \left( \frac{d}{dy} \right)^n e^{- \nu y^2}= n! \sum_{m=0}^{\left[ \frac{n}{2} \right]} \frac{(-1)^n (2 x \nu)^{n-2m}}{m!(n-2m)!}.$$
Putting $t=\sqrt{2}\lambda$ we get 
\begin{equation}\label{w1}
\displaystyle W(x,\lambda)=e^{-\nu \lambda^2+2\nu x \lambda}=\sum_{n=0}^\infty \frac{H_n^\nu(x)}{n!}\lambda^{n}.
\end{equation}

Relabelling $\lambda$ with $t$ we call the function $W(x,t)$ in \eqref{w1} as the generating function of the weighted Hermite polynomials. In order to obtain a recurrence relation, which relates the weighted Hermite polynomials with their consecutive indices, we derive the equation \eqref{w1} with respect to $t$:
$$\displaystyle \frac{\partial W(x,t)}{\partial t}=(-2\nu t+2\nu x)e^{-\nu t^2+2\nu xt}=\sum_{n=1}^{\infty}nH_{n}^{\nu}(x)\frac{t^{n-1}}{n!}. $$

Using another time the generating function \eqref{w1} we obtain 

\begin{equation}
\label{jan1}
-2\nu \sum_{n=0}^{\infty}H_n^\nu(x)\frac{t^{n+1}}{n!}+2\nu x\sum_{n=0}^{\infty}H_n^\nu(x)\frac{t^n}{n!}=\sum_{n=1}^{\infty}H_n^\nu(x)\frac{t^{n-1}}{(n-1)!}.
\end{equation}
By a change of indices in the first sum we get 
$$
\displaystyle -2\nu\sum_{n=1}^{\infty}H^{\nu}_{n-1}(x)\frac{t^n}{(n-1)!}=-2\nu \sum_{n=1}^{\infty}nH_{n-1}^{\nu}(x)\frac{t^n}{n!}.
$$
Thus, by another change of indices in \eqref{jan1} we obtain
$$\displaystyle-2\nu\sum_{n=1}^{\infty}nH_{n-1}^{\nu}(x)\frac{t^n}{n!}+2\nu x\sum_{n=0}^{\infty}H_n^\nu(x)\frac{t^n}{n!}=\sum_{n=0}^{\infty}H_{n+1}^{\nu}(x)\frac{t^n}{n!}.$$

By identifying, the coefficients of $t^n$ we get 
\begin{equation}\label{A1}
H_{n+1}^{\nu}(x)=2\nu x H_n^\nu(x)-2n\nu H_{n-1}^{\nu}(x).
\end{equation}
It is possible to derive another recurrence relation satisfied by the weighted Hermite polynomials. We set $ t= \sqrt{2} \lambda$ as in \eqref{w1} and after we differentiate $W(x,t)$ with respect to $x$ 

$$\displaystyle \frac{\partial}{\partial x} W(x,t)=2\nu t e^{-\nu t^2+2\nu xt}=\sum_{n=0}^{\infty} \left( \frac{d}{dx}H_n^\nu(x) \right)\frac{t^n}{n!}.$$
Using the generating function we obtain

$$\displaystyle 2\nu\sum_{n=0}^{\infty}H_n^\nu(x) \frac{t^{n+1}}{n!}=\sum_{n=1}^{\infty}\left( \frac{d}{dx}H_n^\nu(x) \right)\frac{t^n}{n!}.$$

By a change of variables we can identify the coefficients of $t^n$ to get 
\begin{equation}\label{A2}
\frac{d}{dx}H_n^\nu(x) =2\nu nH_{n-1}^\nu(x).
\end{equation}
\begin{rem}
If we put $\nu=1$ in the formulas \eqref{A1} and \eqref{A2} we recover the classical formulas that can be found in \cite{L}.
\end{rem}

\hspace{4mm}

\noindent
Antonino De Martino,
Dipartimento di Matematica \\ Politecnico di Milano\\
Via Bonardi n.~9\\
20133 Milano\\
Italy

\noindent
\emph{email address}: antonino.demartino@polimi.it\\
\emph{ORCID iD}: 0000-0002-8939-4389

\vspace*{5mm}
\noindent
Kamal Diki,
Dipartimento di Matematica \\ Politecnico di Milano\\
Via Bonardi n.~9\\
20133 Milano\\
Italy

\noindent
\emph{email address}: kamal.diki@polimi.it\\
\emph{ORCID iD}: 0000-0002-4359-7535

\end{document}